\documentclass[12pt]{article}
\usepackage{amsmath,amssymb,amsthm}
\usepackage{hyperref}
\usepackage{datetime}
\usepackage{units}
\usepackage{color}
\usepackage[T1]{fontenc}
\usepackage[utf8]{inputenc}
\usepackage{authblk}
\usepackage{bm,latexsym,mathrsfs,enumerate}
\usepackage{color}

\title{On generalized harmonic numbers, Tornheim double series and linear Euler sums \thanks{%
MSC 2010: 65B10, 11B99}}
\author[]{Kunle Adegoke\thanks{adegoke00@gmail.com}}
\affil{Department of Physics and Engineering Physics, \mbox{Obafemi Awolowo University, Ile-Ife, 220005 Nigeria}}

\theoremstyle{plain}
\numberwithin{equation}{section}
\newtheorem{thm}{THEOREM}[section]

\newtheorem{lemma}[thm]{LEMMA}

\newtheorem{cor}[thm]{COROLLARY}

\begin{document}
\date{}
\maketitle
\begin{abstract}
\noindent Direct links between generalized harmonic numbers, linear Euler sums and Tornheim double series are established in a more perspicuous manner than is found in existing literature. We show that every linear Euler sum can be decomposed into a linear combination of Tornheim double series of the same weight. New closed form evaluations of various Euler sums are presented. Finally certain combinations of linear Euler sums that are reducible to Riemann zeta values are discovered.
\end{abstract}
\tableofcontents

\section{Introduction}
\subsection{Generalized harmonic numbers and linear Euler sums}
Generalized harmonic numbers have a long history, having been studied since the time of Euler.
  The $r^{th}$ generalized harmonic number of order $n$, denoted by $H_{r,n}$ in this paper, is defined by $$H_{r,n}=\sum_{s=1}^r\frac 1{s^n}\,,$$ where $H_{r,1}=H_r$ is the $r^{th}$ harmonic number and $H_{0,n}=0$. The generalized harmonic number converges to the Riemann zeta function, $\zeta(n)$: 
\begin{equation}\label{equ.nmhpf42}
\lim_{r\to\infty}H_{r,n}=\zeta(n),\quad \mathcal{R}[n]>1\,,
\end{equation}
since $\zeta(n)=\sum_{s=1}^\infty {s^{-n}}\,.$

\bigskip

Of particular interest in the study of harmonic numbers is the evaluation of infinite series involving the generalized harmonic numbers, especially linear Euler sums of the type $$E(m,n)=\sum_{\nu=1}^\infty \frac{H_{\nu,m}}{\nu^n}\,.$$ The linear sums can be evaluated in terms of zeta values in the following cases: $m=1$, $m=n$, $m+n$ odd and $m+n=6$~(with $n>1$), (see~\cite{flajolet}).

\bigskip

Evaluation of Euler sums, E(m,n), of odd weight, $m+n$, in terms of $\zeta$~values can be accomplished through Theorem~3.1 of~\cite{flajolet}.

\bigskip

As for the case $m=1$, we have
\begin{thm}[Euler]\label{th.nj1j516} For $n-1\in\mathbb{Z^+}$ holds
\[
2E(1,n)=2\sum_{\nu = 1}^\infty  {\frac{{H_\nu }}{{\nu^n }}}  = (n + 2)\zeta (n + 1) - \sum_{j = 1}^{n - 2} {\zeta (j + 1)\zeta (n - j)}\,. 
\]

\end{thm}

\subsection{Tornheim double series and relation to linear Euler sums}

Tornheim double series, T(r,s,t), is defined by
\[
T(r,s,t) = \sum_{\mu  = 1}^\infty  {\sum_{\nu  = 1}^\infty  {\frac{1}{{\mu ^r \nu ^s (\nu  + \mu )^t }}} } 
\]
and named after Leonard~Tornheim who made a systematic and extended study of the series in a 1950 paper,~\cite{tornheim}. $T(r,s,t)$ has the following basic properties~\cite{huard}:
\begin{subequations}
\begin{equation}\label{equ.qg21457}
T(r,s,t)=T(s,r,t)\,,
\end{equation}
\begin{equation}
T(r,s,t) \mbox{ is finite if and only if } r+t>1, s+t>1 \mbox{ and } r+s+t>2\,,
\end{equation}
\begin{equation}
T(r,s,0)=\zeta(r)\zeta(s)\,,
\end{equation}
\begin{equation}
T(r,0,t)+T(t,0,r)=\zeta(r)\zeta(t)-\zeta(r+t)\,,\quad r\ge 2
\end{equation}
\mbox{and}
\begin{equation}\label{equ.kllr3ob}
T(r,s-1,t+1)+T(r-1,s,t+1)=T(r,s,t),\quad r\ge 1, s\ge 1.
\end{equation}
\end{subequations}

In light of~\eqref{equ.nmhpf42}, the useful identity 
\begin{equation}\label{equ.i4135lr}
\sum_{\nu  = 1}^N {\frac{1}{{(\nu  + \mu )^t }}}  = H_{N + \mu ,t}  - H_{\mu ,t}\,,
\end{equation}
leads to
\begin{equation}\label{equ.k6u5u0q}
\sum_{\nu  = 1}^\infty  {\frac{1}{{(\nu  + \mu )^t }}}  = \zeta (t) - H_{\mu ,t}\,,
\end{equation}
which establishes the link between the Hurwitz zeta function, $\zeta(t,\mu)$, the Riemann zeta function and the generalized harmonic numbers (see also equation~\mbox{(1.19)} of~\cite{adamchik}) as
\begin{equation}
\zeta(t,\mu)=\zeta(t)-H_{\mu-1,t}\,,
\end{equation}
since
$$\zeta(t,\mu)=\sum_{\nu=0}^\infty\frac 1{(\mu+\nu)^t}\,.$$

\bigskip

The identity~\eqref{equ.k6u5u0q} also brings out the direct connection between the linear Euler sums and the Tornheim double series, namely,
\begin{equation}\label{equ.lf9ruko}
E(n,m)=\zeta(n)\zeta(m)-T(m,0,n),\quad n>1,m>1\,.
\end{equation}
Differentiating the identity
\begin{equation}\label{equ.yav4944}
\frac{1}{\nu } - \frac{1}{{\nu  + \mu }}=\frac{\mu }{{\nu (\nu  + \mu )}}\,,
\end{equation}

$n-1$ times with respect to $\nu$ gives
\begin{equation}\label{equ.gk1txsj}
\frac{1}{{\nu ^n }} - \frac{1}{{(\nu  + \mu )^n }}=\sum_{p = 0}^{n - 1} {\frac{\mu }{{\nu ^{p + 1} (\nu  + \mu )^{n - p} }}}\,,\quad n\in\mathbb{N}_0\,,
\end{equation}
from which, by summing over $\nu$, employing~\eqref{equ.i4135lr}, we obtain

\[
H_{N,n}  - H_{N + \mu ,n}  + H_{\mu ,n}=\sum_{p = 0}^{n - 1} {\sum_{\nu  = 1}^N {\frac{\mu }{{\nu ^{p + 1} (\nu  + \mu )^{n - p} }}} }\,, 
\]
and hence, in the limit $N\to\infty$ we have 
\begin{equation}\label{equ.mvt7nzf}
H_{\mu ,n}=\sum_{p = 0}^{n - 1} {\sum_{\nu  = 1}^\infty  {\frac{\mu }{{\nu ^{p + 1} (\nu  + \mu )^{n - p} }}} }\,,\quad n\in\mathbb{Z^+}\,,
\end{equation}
from which follows, for $n,r-1\in\mathbb{Z^+}$, the interesting relation
\[
\sum_{\mu  = 1}^\infty  {\frac{{H_{\mu ,n} }}{{\mu ^r }}}=\sum_{p = 0}^{n - 1} {\sum_{\mu  = 1}^\infty  {\sum_{\nu  = 1}^\infty  {\frac{1}{{\mu ^{r - 1} \nu ^{p + 1} (\nu  + \mu )^{n - p} }}} } }\,, 
\]
that is,
\[
 E(n,r)=\sum_{p = 0}^{n - 1} {T(r - 1,p + 1,n - p)}\,,
\]
so that,
\begin{thm}\label{thm.tiqlu9p}
Any linear Euler sum can be decomposed into a linear combination of Tornheim double series:
\[
\boxed{E(n,r)=\sum_{p=1}^n T(r-1,n-p+1,p)}\,,\quad n,r-1\in\mathbb{Z^+}\,.
\]
\end{thm}
\begin{cor}
$$E(1,r)=T(r-1,1,1),\quad r>1\,.$$
\end{cor}

Note the use of the index shift identity:
\begin{equation}\label{equ.mjf11t4}
\sum_{i = a}^b {f_i}\equiv\sum_{i = u - b}^{u - a} {f_{u - i}}\,,
\end{equation}
of which the not so familiar index shift formulas (see~\cite{gould}) $$\sum_{i=a}^b{f_i}=\sum_{i=0}^{b-a}{f_{b-i}}$$ and $$\sum_{i=a}^b{f_i}=\sum_{i=a}^{b}{f_{a+b-i}}$$ are particular cases, obtained, respectively, by setting $u=b$ and $u=a+b$.

\bigskip

A useful variant of~\eqref{equ.mjf11t4} is
\[
\sum_{i = a}^{b - a} {f_i}\equiv\sum_{i = a}^{b - a} {f_{b-i}}\,.
\]

\bigskip

Throughout this paper we shall make frequent tacit use of the index shift identity. As an immediate application, if we choose $f_\nu=1/(\mu-\nu)^t$, $a=1$, $b=N$ and $u=\mu$ in~\eqref{equ.mjf11t4} we obtain the following analog of~\eqref{equ.i4135lr}:
\begin{equation}\label{equ.s2fevcr}
\sum_{\nu  = 1}^N {\frac{1}{{(\mu  - \nu )^t }}}  = H_{\mu  - 1,t}  - H_{\mu  - N - 1,t}\,,
\end{equation}
which, over the ring of integers, is valid for \mbox{$1\le N<\mu$}.

\bigskip

For evaluating sums with $a>b$ we shall use
\[
\sum_{i=a}^b{f_i}\equiv-\sum_{i=b+1}^{a-1}{f_i}\,.
\]
In particular $$\sum_{i=a+1}^{a-1}{f_i}=-f_a\mbox{ and }\sum_{i=a}^{a-1}{f_i}=0\,.$$

\bigskip

The beautiful formula from case $n=1$ in \eqref{equ.mvt7nzf}, that is,
\begin{equation}\label{equ.gbbnxxa}
\sum_{\nu=1}^\infty\frac\mu{\nu(\mu+\nu)}=H_\mu\,,\quad \mu\in\mathbb{C}\backslash\mathbb{Z^-}\,,
\end{equation}
was also derived in~\cite{georghiou} and~\cite{basu}.

\bigskip

A brief summary of some of the interesting results contained in this paper is given in Section~\ref{sec.summary}.

\section{Generalized harmonic numbers and \mbox{summation} of series}
In this section we discuss the evaluation of certain sums in terms of the Riemann zeta function and the generalized harmonic numbers.

\bigskip

When the following partial fraction decomposition, valid for $\mu,\nu\in\mathbb{C}\backslash\{0\}$ and $s,t\in\mathbb{Z}$ such that either $st\in\mathbb{Z^-}$ or $s+t\in\mathbb{Z^+}$ (see for example equation (2.4) of~\cite{subbarao} ),
\begin{equation}\label{equ.hrp8apz}
\frac{1}{{\nu ^s (\nu  + \mu )^t }} = \sum_{i = 0}^{s - 1} {\binom{t+i-1}i\frac{{( - 1)^i }}{{\nu ^{s - i} \mu ^{t + i} }}}  + \sum_{i = 0}^{t - 1} {\binom{s+i-1}i\frac{{( - 1)^s }}{{\mu ^{s + i} (\nu  + \mu )^{t - i} }}}\,,
\end{equation}
is summed over $\nu$, taking~\eqref{equ.i4135lr} into consideration, we get
\begin{thm}\label{thm.b2i7lfz}
For $\mu\in\mathbb{C}\backslash\mathbb{Z^-}$, $\mu\ne 0$ and $s,t\in\mathbb{Z}$ such that either $st\in\mathbb{Z^-}$ or $s+t\in\mathbb{Z^+}$ holds
\[
\begin{split}
\sum_{\nu  = 1}^N {\frac{1}{{\nu ^s (\nu  + \mu )^t }}}  &= \sum_{i = 0}^{s - 1} {\binom{t+i-1}i\frac{{( - 1)^i H_{N,s - i} }}{{\mu ^{t + i} }}}\\
&\qquad  + ( - 1)^s\sum_{i = 0}^{t - 1} {\binom{s+i-1}i\frac{{ \left[ {H_{N + \mu ,t - i}  - H_{\mu ,t - i} } \right]}}{{\mu ^{s + i} }}}\,.
\end{split}
\]

In particular,
\begin{equation}\label{equ.s45a41o}
\begin{split}
\sum_{\nu  = 1}^N {\frac{1}{{\nu ^s (\nu  + \mu )^s }}}  &= \sum_{i = 0}^{s - 1} {\binom{s+i-1}i\frac{{(-1)^iH_{N,s - i}  + (-1)^s\left[H_{N+\mu,s - i}  - H_{\mu,s - i}\right] }}{{\mu ^{s + i} }}}\\
&= ( - 1)^{s - 1} \sum_{i = 1}^s {\binom{2s-i-1}{s-1}\frac{{H_{\mu ,i}  - H_{N + \mu ,i}  - ( - 1)^i H_{N,i} }}{{\mu ^{2s - i} }}}\,.
\end{split}
\end{equation}

\end{thm}

Writing \eqref{equ.hrp8apz} as
\begin{equation}\label{equ.p5bh6gf}
\begin{split}
\frac{1}{{\nu ^s (\nu  + \mu )^t }} &= \sum_{i = 0}^{s - 2} {\binom{t+i-1}i\frac{{( - 1)^i }}{{\nu ^{s - i} \mu ^{t  + i} }}}\\
&\quad  + ( - 1)^s \sum_{i = 0}^{t - 2} {\binom{s+i-1}i\frac{{1}}{{\mu ^{s + i} (\nu  + \mu )^{t - i} }}}\\
&\qquad+ ( - 1)^{s - 1} \binom{s+t-2}{s-1}\frac{1}{{\mu ^{ s + t - 1} }}\frac{\mu }{{\nu (\nu  + \mu )}}\,,
\end{split}
\end{equation}

and summing from $\nu=1$ to $\nu=\infty$ gives
\begin{cor}\label{thm.p1iwzvm}
For $\mu\in\mathbb{C}\backslash\mathbb{Z^-}$, $\mu\ne 0$ and $s,t\in\mathbb{Z^+}$ holds
\[
\begin{split}
\sum_{\nu  = 1}^\infty  {\frac{1}{{\nu ^s (\nu  + \mu )^t }}}  &= \sum_{i = 0}^{s - 2} {\binom{t+i-1}i\frac{{( - 1)^i \zeta (s - i)}}{{\mu ^{t + i} }}}\\
&\quad+ ( - 1)^s\sum_{i = 0}^{t - 2} {\binom{s+i-1}i\frac{{ \left[ {\zeta (t - i) - H_{\mu ,t - i} } \right]}}{{\mu ^{s + i} }}}\\
&\qquad + ( - 1)^{s - 1} \binom{s+t-2}{s-1}\frac{{H_\mu  }}{{\mu ^{s + t - 1} }}\,.
\end{split}
\]
In particular,
\begin{equation}\label{equ.cgxlioc}
\begin{split}
\sum_{\nu  = 1}^\infty  {\frac{1}{{\nu ^s (\nu  + \mu )^s }}} &= ( - 1)^{s - 1} \sum_{i = 1}^s {\binom{2s-i-1}{s-1}\frac{{H_{\mu ,i} }}{{\mu ^{2s - i} }}}\\
&\qquad- ( - 1)^{s - 1} 2\sum_{i = 1}^{\left\lfloor {{s \mathord{\left/
 {\vphantom {s 2}} \right.
 \kern-\nulldelimiterspace} 2}} \right\rfloor } {\binom{2s-2i-1}{s-1}\frac{{\zeta (2i)}}{{\mu ^{2s - 2i} }}}\,.
\end{split}
\end{equation}

\end{cor}

Substituting $-\nu$ for $\nu$ in~\eqref{equ.gk1txsj}, summing over $\nu$ and making use of~\eqref{equ.s2fevcr}, we obtain
\begin{thm}
For \mbox{$1\le N<\mu$} and $n\in\mathbb{N}_0$ holds
\[
\sum_{p = 1}^n {\left\{ {( - 1)^{p - 1} \sum_{\nu  = 1}^N {\frac{\mu }{{\nu ^{n - p + 1} (\mu  - \nu )^p }}} } \right\}}=H_{N,n}  - ( - 1)^n \left[ {H_{\mu  - 1,n}  - H_{\mu  - N - 1,n} } \right]\,.
\]

In particular,
\begin{equation}
\sum_{\nu  = 1}^N {\frac{\mu }{{\nu (\mu  - \nu )}}}  = H_N  + H_{\mu  - 1}  - H_{\mu  - N - 1}\,.
\end{equation}

\end{thm}

\bigskip

Replacing $\nu$ by $-\nu$ in~\eqref{equ.hrp8apz} gives the identity
\[
\frac{1}{{\nu ^s (\mu  - \nu )^t }} = \sum_{i = 0}^{s - 1} {\binom{t+i-1}i\frac{1}{{\nu ^{s - i} \mu ^{t + i} }}}  + \sum_{i = 0}^{t - 1} {\binom{s+i-1}i\frac{1}{{\mu ^{s + i} (\mu  - \nu )^{t - i} }}}\,,
\]
from which we obtain, after summing over $\nu$, using~\eqref{equ.s2fevcr},
\begin{thm}\label{thm.u9a6spe}
For \mbox{$1\le N<\mu$} and $s,t\in\mathbb{Z}$ such that either $st\in\mathbb{Z^-}$ or $s+t\in\mathbb{Z^+}$ holds
\[
\begin{split}
\sum_{\nu  = 1}^N {\frac{1}{{\nu ^s (\mu  - \nu )^t }}}  &= \sum_{i = 0}^{s - 1} {\binom{t+i-1}i\frac{{H_{N,s - i} }}{{\mu ^{t + i} }}}\\
&\qquad+ \sum_{i = 0}^{t - 1} {\binom{s+i-1}i\frac{{\left[ {H_{\mu  - 1,t - i}  - H_{\mu  - N - 1,t - i} } \right]}}{{\mu ^{s + i} }}}\\
&\\
&\quad\equiv \sum_{i = 1}^{s} {\binom{s+t-i-1}{t-1}\frac{{H_{N,i} }}{{\mu ^{s+t-i} }}}\\
&\qquad+ \sum_{i = 1}^{t} {\binom{s+t-i-1}{s-1}\frac{{\left[ {H_{\mu  - 1,i}  - H_{\mu  - N - 1,t} } \right]}}{{\mu ^{s + t - i} }}}\,.
\end{split}
\]

In particular,
\begin{equation}\label{equ.mzn7m04}
\sum_{\nu  = 1}^N {\frac{1}{{\nu ^s (\mu  - \nu )}}}  = \frac{{H_{\mu  - 1}  - H_{\mu  - N - 1} }}{{\mu ^s }} + \sum_{i = 1}^{s} {\frac{{H_{N,i} }}{{\mu ^{s - i + 1} }}}\,,
\end{equation}

\begin{equation}\label{equ.pjzwe2i}
\sum_{\nu  = 1}^N {\frac{1}{{\nu (\mu  - \nu )^t }}}  = \frac{{H_N }}{{\mu ^t }} + \sum_{i = 1}^{t} {\frac{{H_{\mu  - 1,i}  - H_{\mu  - N - 1,i} }}{{\mu ^{t - i + 1} }}}
\end{equation}
and
\begin{equation}\label{equ.hega9ix}
\sum_{\nu  = 1}^N {\frac{1}{{\nu ^s (\mu  - \nu )^s }}}  = \sum_{i = 1}^{s} {\binom{2s-i-1}{s-1}\frac{{H_{N,i}  + H_{\mu  - 1,i}  - H_{\mu  - N - 1,i} }}{{\mu ^{2s - i} }}}\,.
\end{equation}

\end{thm}

\bigskip

\begin{lemma}\label{thm.amfvtaz}
Let $a$, $c$ and $f$ be arbitrary functions such that $a\ne 0$, $c\ne 0$ and $af=c+a$, then for $m\in\mathbb{Z^+}$ holds $$af^{m}=af+\sum_{i=1}^{m-1}cf^i\,,$$
\end{lemma}
or, equivalently, using the index shift identity,
$$af^{m}=af+\sum_{i=0}^{m-2}cf^{m-i-1}\,.$$

The Lemma is easily proved by the application of mathematical induction on $m$.

\bigskip

Choosing $a=-1/(\mu+\nu)$, $c=1/\nu$ and $f=-\mu/\nu$ in Lemma~\ref{thm.amfvtaz} gives the partial fraction decomposition
\begin{equation}\label{equ.xbugcw0}
( - 1)^{m - 1} \frac{{\mu ^m }}{{\nu ^m (\nu  + \mu )}} = \frac{\mu }{{\nu (\nu  + \mu )}} + \sum_{i = 1}^{m - 1} {( - 1)^i \frac{{\mu ^i }}{{\nu ^{i + 1} }}}\,,
\end{equation}
which, after $n$ times differentiation with respect to $\mu$, yields, for \mbox{$m,n\in\mathbb{Z}$} such that $mn\in\mathbb{Z^-}$ or $m+n\in\mathbb{N}_0$, the identity
\begin{equation}\label{equ.tfna05t}
\begin{split}
&( - 1)^{m - 1} \sum_{p = 0}^n {\left\{ {( - 1)^p \binom mp\frac{{\mu ^{m - p} }}{{\nu ^m (\nu  + \mu )^{n - p + 1} }}} \right\}}\\
&\\
&\qquad =  - \frac{1}{{(\nu  + \mu )^{n + 1} }} + ( - 1)^n \sum_{i = n}^{m - 1} {\left\{ {( - 1)^i \binom in\frac{{\mu ^{i - n} }}{{\nu ^{i + 1} }}} \right\}}\,,
\end{split}
\end{equation}
from which upon summing from $\nu=1$ to $\nu=N$ we have
\begin{thm}\label{thm.x1nstmi}
For \mbox{$n,m\in\mathbb{Z}$} such that either $nm\le 0$ or \mbox{$n+m\in\mathbb{N}_0$} and $\mu\in\mathbb{C}\backslash\mathbb{Z^-}$ holds
\[
\begin{split}
&( - 1)^{m - 1} \sum_{p = 0}^n {\left\{ {( - 1)^p\binom mp \sum_{\nu  = 1}^N {\frac{{\mu ^{m - p} }}{{\nu ^m (\nu  + \mu )^{n - p + 1} }}} } \right\}}\\
&\\
&\qquad= H_{\mu ,n + 1}  - H_{N+\mu,n + 1}  + ( - 1)^n \sum_{i = n}^{m - 1} {\left\{ {( - 1)^i\binom in \mu^{i-n}H_{N,i + 1} } \right\}}\,,
\end{split}
\]

\end{thm}
which in the limit $N\to\infty$ gives
\begin{cor}\label{thm.s2t7fyj}
For \mbox{$m\in\mathbb{Z^+}$} and \mbox{$n\in\mathbb{N}_0$}, $\mu\in\mathbb{C}\backslash\mathbb{Z^-}$ holds
\[
\begin{split}
&( - 1)^{m - 1} \sum_{p = 0}^n {\left\{ {( - 1)^p\binom mp \sum_{\nu  = 1}^\infty {\frac{{\mu ^{m - p} }}{{\nu ^m (\nu  + \mu )^{n - p + 1} }}} } \right\}}\\
&\\
&\qquad= H_{\mu ,n + 1}  + ( - 1)^n \sum_{i = n+1}^{m - 1} {\left\{ {( - 1)^i\binom in \mu^{i-n}\zeta({i + 1}) } \right\}}\,.
\end{split}
\]

In particular, for $m,n\in\mathbb{Z^+}$ and $\mu\in\mathbb{C}\backslash\mathbb{Z^-}$, we have
\begin{equation}\label{equ.k9w7fxv}
( - 1)^{m - 1}  \sum_{\nu  = 1}^\infty {\frac{ \mu ^m}{{\nu ^m (\nu  + \mu )}}}  =  H_\mu   + \sum_{i = 1}^{m - 1} {( - 1)^i \mu ^i \zeta(i + 1) }\,,
\end{equation}

\[
\sum_{p = 0}^n {\left\{ {( - 1)^p\binom np \sum_{\nu  = 1}^\infty {\frac{{\mu ^p }}{{\nu ^n (\nu  + \mu )^{p + 1} }}} } \right\}}=\zeta(n+1)-H_{\mu ,n + 1}
\]
and
\[
\sum_{p = 1}^n {\left\{ {( - 1)^{p-1} \binom np\sum_{\nu  = 1}^\infty  {\frac{{\mu ^p }}{{\nu ^n (\nu  + \mu )^p }}} } \right\}}  = H_{\mu ,n}\,. 
\]

\end{cor}

\bigskip

Replacing $\nu$ by $-\nu$ in~\eqref{equ.tfna05t} and summing over $\nu$ gives
\begin{thm}
For \mbox{$m\in\mathbb{Z}$} and \mbox{$n\in\mathbb{N}_0$}, \mbox{$1\le N<\mu$} holds
\[
\begin{split}
&\sum_{p = 0}^n {\left\{ {( - 1)^p\binom mp \sum_{\nu  = 1}^N {\frac{{\mu ^{m - p} }}{{\nu ^m (\mu  - \nu )^{n - p + 1} }}} } \right\}}\\
&\\
&\qquad= H_{\mu-1 ,n + 1}  - H_{\mu-N-1,n + 1}  + ( - 1)^n \sum_{i = n}^{m - 1} {\left\{ {\binom in \mu^{i-n}H_{N,i + 1} } \right\}}\,.
\end{split}
\]

\end{thm}
Differentiating~\eqref{equ.xbugcw0} $n$ times with respect to $\nu$, we obtain, for $m\in\mathbb{Z^+}$ and $n\in\mathbb{N}_0$, the identity
\begin{equation}\label{equ.amxnxuz}
\begin{split}
&( - 1)^m \sum_{p = 0}^n {\left\{ {\binom{m+p-1}p\frac{{\mu ^{m + 1} }}{{\nu ^{m + p} (\nu  + \mu )^{n - p + 1} }}} \right\}}\\
&\\
&\qquad = \frac{\mu }{{(\nu  + \mu )^{n + 1} }} + \sum_{i = 1}^m {\left\{ {( - 1)^i \binom{i+n-1}n\frac{{\mu ^i }}{{\nu ^{i + n} }}} \right\}}\,,
\end{split}
\end{equation}
from which upon summing from $\nu=1$ to $\nu=N$ we have
\begin{thm}
For $m\in\mathbb{Z}$, $\mu\in\mathbb{C}\backslash\mathbb{Z^-}$ and $n\in\mathbb{N}_0$ holds
\[
\begin{split}
&( - 1)^m \sum_{p = 0}^n {\left\{ {\binom{m+p-1}p\sum_{\nu  = 1}^N {\frac{{\mu ^{m + 1} }}{{\nu ^{m + p} (\nu  + \mu )^{n - p + 1} }}} } \right\}}\\
&\quad = \mu H_{N + \mu ,n + 1}  - \mu H_{N,n + 1}\\
&\qquad - \mu H_{\mu ,n + 1}\\
&\qquad\quad + \sum_{i = 2}^m {( - 1)^i \binom{i+n-1}n\mu^iH_{N,i + n} }\,,
\end{split}
\]

\end{thm}
which in the limit $N\to\infty$ gives
\begin{cor}\label{thm.ztqq3ah}
For $m\in\mathbb{Z^+}$, $\mu\in\mathbb{C}\backslash\mathbb{Z^-}$ and $n\in\mathbb{N}_0$ holds
\[
\begin{split}
&( - 1)^m \sum_{p = 0}^n {\left\{ {\binom{m+p-1}p\sum_{\nu  = 1}^\infty  {\frac{{\mu ^{m + 1} }}{{\nu ^{m + p} (\nu  + \mu )^{n - p + 1} }}} } \right\}}\\
&\\
&\qquad =  - \mu H_{\mu ,n + 1}  + \sum_{i = 2}^m {( - 1)^i \binom{i+n-1}n\mu ^i \zeta (n + i)}\,.
\end{split}
\]

\end{cor}

\bigskip

Changing $\nu$ to $-\nu$ in~\eqref{equ.amxnxuz} and summing over $\nu$ gives
\begin{thm}
For $m\in\mathbb{Z}$ and \mbox{$n\in\mathbb{N}_0$}, \mbox{$1\le N<\mu$} holds
\[
\begin{split}
&\sum_{p = 0}^n {\left\{ {\binom{m+p-1}p(-1)^p\sum_{\nu  = 1}^N {\frac{{\mu ^{m + 1} }}{{\nu ^{m + p} (\mu  - \nu )^{n - p + 1} }}} } \right\}}\\
&\quad = \mu H_{\mu - 1 ,n + 1}  - \mu H_{\mu - N - 1,n + 1}\\
&\qquad\quad + (-1)^n\sum_{i = 1}^m {\binom{i+n-1}n\mu^iH_{N,i + n} }\,.
\end{split}
\]

\end{thm}
\bigskip

Using $a=1/{\mu\nu}$, $c=-1/\left({\mu(\mu+\nu)}\right)$ and $f=\mu/{(\mu+\nu)}$ in Lemma~\ref{thm.amfvtaz} gives the identity
\[
\frac{{\mu ^{m} }}{{\nu (\nu  + \mu )^{m} }} = \frac 1\nu- \sum_{i = 1}^m {\frac{{\mu ^{i-1} }}{{(\nu  + \mu )^{i} }}}\,, 
\]
which, after $n$ differentiations with respect to $\nu$, gives
\begin{equation}\label{equ.pkqyaem}
\begin{split}
&\sum_{p = 0}^n {\binom{m+n-p-1}{m-1}\frac{{\mu ^m }}{{\nu ^{p + 1} (\nu  + \mu )^{m + n - p} }}}\\
&\qquad  = \frac{1}{{\nu ^{n + 1} }} - \sum_{i = 1}^m {\binom{i+n-1}n\frac{{\mu ^{i - 1} }}{{(\nu  + \mu )^{i + n} }}}\,,
\end{split}
\end{equation}
from which we get, after summing over $\nu$,
\begin{thm}
For $m\in\mathbb{Z^+}$, $\mu\in\mathbb{C}\backslash\mathbb{Z^-}$ and $n\in\mathbb{N}_0$ holds
\[
\begin{split}
&\sum_{p = 0}^n {\left\{ {\binom{m+n-p-1}{m-1}\sum_{\nu  = 1}^N {\frac{{\mu ^m }}{{\nu ^{p + 1} (\nu  + \mu )^{m + n - p} }}} } \right\}}\\
&\quad\qquad= H_{N,n + 1}  - \sum_{i = 1}^m {\binom{i+n-1}n\mu ^{i - 1} H_{N + \mu ,i + n} }  + \sum_{i = 1}^m {\binom{i+n-1}n\mu ^{i - 1} H_{\mu ,i + n} }\,,
\end{split}
\]
\end{thm}
which in the limit $N\to\infty$ gives 
\begin{cor}\label{thm.cwe2ocv}
For $\mu\in\mathbb{C}\backslash\mathbb{Z^-}$ and $m,n\in\mathbb{N}_0$ holds
\[
\begin{split}
&\sum_{p = 0}^n {\left\{ {\binom{m+p}{m}\sum_{\nu  = 1}^\infty {\frac{{\mu ^{m+1} }}{{\nu ^{n-p + 1} (\nu  + \mu )^{m + p +1} }}} } \right\}}\\
&\quad\qquad=\sum_{i = n}^{m+n} {\binom in\mu ^{i - n} H_{\mu ,i + 1} }  -\sum_{i = n+1}^{m+n} {\binom in\mu ^{i - n} \zeta(i + 1) }\,.
\end{split}
\]
\end{cor}

In particular, for $m\in\mathbb{N}_0$ and $\mu\in\mathbb{C}\backslash\mathbb{Z^-}$, we have
\begin{equation}\label{equ.gqhxhjb}
\sum_{\nu  = 1}^\infty  {\frac{\mu ^{m + 1}}{{\nu (\nu  + \mu )^{m + 1} }}}  = H_\mu   - \sum_{i = 1}^m {\mu ^i \zeta (i + 1)}  + \sum_{i = 1}^m {\mu ^i H_{\mu ,i + 1} }\,.
\end{equation}

\bigskip

Replacing $\nu$ by $-\nu$ in~\eqref{equ.pkqyaem} and summing over $\nu$ we obtain
\begin{thm}
For $m\in\mathbb{Z}$ and \mbox{$n\in\mathbb{N}_0$}, \mbox{$1\le N<\mu$} holds
\[
\begin{split}
&\sum_{p=0}^n{\left\{(-1)^p\binom{m+n-p-1}{m-1}\sum_{\nu=1}^N{\frac {\mu^m}{\nu^{p+1}(\mu-\nu)^{m+n-p}}}\right\}}\\
&\quad=(-1)^nH_{N,n+1}+\sum_{i=1}^m{\binom{i+n-1}n\mu^{i-1}\left(H_{\mu-1,i+n}-H_{\mu-N-1,i+n}\right)}\,.
\end{split}
\]

\end{thm}
\section{Functional relations for the \mbox{Tornheim double series}}\label{sec.functional}
In this section we will derive various functional relations for the Tornheim double series and from these we will see that it is already possible to evaluate the series at certain arguments.

\bigskip

Writing~\eqref{equ.yav4944} as $$\frac 1{\mu+\nu}+\frac\mu{\nu(\mu+\nu)}=\frac 1\nu\,,$$ dividing through by \mbox{$\mu^r\nu^{s-1}(\mu+\nu)^t$} and summing over $\mu$ and $\nu$ gives 
\begin{equation}\label{equ.shguxqu}
T(r,s-1,t+1)+T(r-1,s,t+1)=T(r,s,t)\,,
\end{equation}
which is property~\eqref{equ.kllr3ob} of Tornheim series.

\bigskip

Replacing $\mu$ by $\mu-\nu$ in~\eqref{equ.hrp8apz} and then replacing $\nu$ by $-\nu$ in the resulting identity gives the following variant of~\eqref{equ.hrp8apz}, valid for $\mu,\nu\in\mathbb{C}\backslash\{0\}$ and \mbox{$s,t\in\mathbb{Z}$} such that either $st\in\mathbb{Z^-}$ or $s+t\in\mathbb{Z^+}$,
\[
\frac{1}{{\nu ^s \mu ^t }} = \sum_{i = 0}^{s - 1} {\binom{t+i-1}i\frac{1}{{\nu ^{s - i} (\mu  + \nu )^{t + i} }}}  + \sum_{i = 0}^{t - 1} {\binom{s+i-1}i\frac{1}{{\mu ^{t - i} (\mu  + \nu )^{s + i} }}}\,, 
\]
from which, after dividing through by $(\mu+\nu)^r$ and summing over $\mu$ and $\nu$, we get
\begin{thm}\label{equ.a1ft4e9}
For $r+s>1$, $r+t>1$ and $r+s+t>2$ holds
\begin{equation}
\begin{split}
T(s,t,r) &= \sum_{i = 0}^{s - 1} {\binom{t+i-1}iT(s - i,0,t + r + i)}  + \sum_{i = 0}^{t - 1} {\binom{s+i-1}iT(t - i,0,s + r + i)}\\
&\quad= \sum_{i = 1}^s {\binom{s+t-i-1}{t-1}T(i,0,s + r + t - i)}\\
&\qquad  + \sum_{i = 1}^t {\binom{s+t-i-1}{s-1}T(i,0,s + r + t - i)}\,,
\end{split}
\end{equation}

\end{thm} 
without the need of the induction suggested in~\cite{huard}, where the identity of the theorem was first established.

\bigskip

The utility of Theorem~\ref{equ.a1ft4e9} lies in the fact that in discussing $T(r,s,t)$, it is sufficient to consider only the set $\left\{T(i,0,r+s+t-i)\right\}$ for $1\le i\le\max{\{s,t\}}$; a fact that was gainfully employed in~\cite{huard} and~\cite{espinosa2}.

\bigskip

Dividing through the identity of Corollary~\ref{thm.s2t7fyj} by $\mu^r$ and summing over $\mu$ we obtain
\begin{thm}\label{thm.isoq1ou}
For \mbox{$m,r-1\in\mathbb{Z^+}$}, \mbox{$n\in\mathbb{N}_0$} and $r>m-n$ holds
\[
\begin{split}
&( - 1)^{m - 1} \sum_{p = 0}^n {\left\{ {( - 1)^p \binom mpT(r - m + p,m,n - p + 1)} \right\}}\\
&\qquad= E(n + 1,r) + ( - 1)^n \sum_{i = n + 1}^{m - 1} {\left\{ {( - 1)^i \binom in\zeta (i + 1)\zeta (r - i + n)} \right\}}\,.
\end{split}
\]

\end{thm}
Setting $m=n+1$ and replacing $n+1$ by $n$, we obtain
\begin{cor}
\[\boxed{E(n,r)=\sum_{p=1}^n{(-1)^{p-1}\binom npT(r-p,n,p)}}\,,\]
\end{cor}
which is the alternating version of Theorem~\ref{thm.tiqlu9p}.

\bigskip

Setting $n=0$ in the identity of Theorem~\ref{thm.isoq1ou} and using Theorem~\ref{th.nj1j516} to write $E(1,r)$ in zeta values we have
\begin{cor}
For $m,r-1\in\mathbb{Z^+}$ and $r>m$ holds
\[
\begin{split}
&( - 1)^{m - 1} T(r - m,m,1)\\
&\qquad= \frac{1}{2}(r + 2)\zeta (r + 1) - \frac{1}{2}\sum_{i = 1}^{r - 2} {\zeta (r - i)\zeta (i + 1)}\\
&\qquad\quad+ \sum_{i = 1}^{m - 1} {( - 1)^i \zeta (r - i)\zeta (i + 1)}\,.
\end{split}
\]
\end{cor}

Setting $n=1$ in the identity of Theorem~\ref{thm.isoq1ou} and employing the above corollary gives
\begin{cor}\label{thm.fl7eehk}
For $r-1\in\mathbb{Z^+}$, $m\in\mathbb{N}_0$ and $r>m-1$ holds
\[
\begin{split}
&\quad( - 1)^{m - 1} T(r - m,m,2)\\
&\qquad= E(2,r) + \frac{m}{2}(r + 3)\zeta (r + 2) - \frac{m}{2}\sum_{i = 1}^{r - 1} {\zeta (r - i + 1)\zeta (i + 1)}\\
&\quad\qquad+ m\sum_{i = 1}^{m - 1} {( - 1)^i \zeta (r - i + 1)\zeta (i + 1)}  - \sum_{i = 2}^{m - 1} {( - 1)^i i\zeta (r - i + 1)\zeta (i + 1)}\,.
\end{split}
\]
Setting $r=4$ in Corollary~\ref{thm.fl7eehk}, noting that $E(2,4)=\zeta(3)^2-\zeta(6)/3$ (see~\cite{flajolet}, page~16), and using $m=0$, $m=1$ and $m=2$, respectively, we have
$$T(4,0,2)=\frac{\zeta(6)}3+\zeta(4)\zeta(2)-\zeta(3)^2\,,$$
$$T(3,1,2)=\frac{\zeta(3)^2}2+\frac{19}6\zeta(6)-2\zeta(4)\zeta(2)$$
and
\begin{equation}\label{equ.aaxvloz}
T(2,2,2)=-\frac{20}3\zeta(6)+4\zeta(4)\zeta(2)\,.
\end{equation}

\bigskip

Similarly, setting $r=2$ in Corollary~\ref{thm.fl7eehk}, noting that \mbox{$E(2,2)=\zeta(2)^2/2+\zeta(4)/2$}, \mbox{(see~\eqref{equ.csxx7lv})}, and using $m=0$, $m=1$, respectively, we have $$T(2,0,2)=\frac{\zeta(2)^2}2-\frac12\zeta(4)$$ and 
\begin{equation}\label{equ.fdkkbct}
T(1,1,2)=-\zeta(2)^2+3\zeta(4)\,.
\end{equation}
\end{cor}

\bigskip

Setting $r=s$ in~\eqref{equ.shguxqu} and using the reflection property gives
\begin{equation}
2T(s,s-1,t+1)=T(s,s,t)\,.
\end{equation}
In particular, we have $$2T(s,s-1,1)=T(s,s,0)=\zeta(s)^2\,,\quad s-1\in\mathbb{Z^+}$$ and $$2T(s,s-1,3)=T(s,s,2)\,,\quad s\in\mathbb{Z^+}\,,$$ from which we get $$2T(1,0,3)=-\zeta(2)^2+3\zeta(4)$$ and $$2T(2,1,3)=-\frac{20}3\zeta(6)+4\zeta(4)\zeta(2)\,,$$ after using~\eqref{equ.fdkkbct} and \eqref{equ.aaxvloz}.

\bigskip

Setting $r=n+1$ in the identity of Theorem~\ref{thm.isoq1ou} we obtain
\begin{cor}\label{thm.ecbg7m6}
For \mbox{$m,n\in\mathbb{Z^+}$} and $2n>m-1$ holds
\[
\begin{split}
&( - 1)^{m - 1} 2\sum_{p = 0}^n {\left\{ {( - 1)^p \binom mpT(n - m + p+1,m,n - p + 1)} \right\}}\\
&\qquad= \zeta(n+1)^2+\zeta(2n+2)\\
&\qquad\qquad + ( - 1)^n 2\sum_{i = n + 1}^{m - 1} {\left\{ {( - 1)^i \binom in\zeta (i + 1)\zeta (2n - i + 1)} \right\}}\,.
\end{split}
\]
\end{cor}

\bigskip

Setting $m=n+1$ in Corollary~\ref{thm.ecbg7m6}, substituting $n$ for $n+1$ and utilizing the index shift identity~\eqref{equ.mjf11t4} gives, for $n-1\in\mathbb{Z^+}$, 
\[
2\sum_{p=1}^n(-1)^{p-1}\binom np T(n-p,n,p)=\zeta(n)^2+\zeta(2n)\,,
\]
from which, with the aide of Corollary~\ref{cor.p5c00dq} and after some manipulation, we get,
\[
\sum_{p = 1}^{2n - 1} {( - 1)^{p - 1} \binom{2n}p T(2n - p,2n,p)}  = \zeta (2n)^2\,,\quad n\in\mathbb{Z^+} 
\]
and
\[
\begin{split}
&\sum_{p = 1}^{2n} {( - 1)^{p - 1} \binom{2n+1}pT(2n - p + 1,2n + 1,p)}\\
&\qquad\qquad\qquad  = \zeta \left( {2(2n + 1)} \right)\,,\quad n\in\mathbb{Z^+}\,.
\end{split}
\]
In particular,
$$2T(1,2,1)=\zeta(2)^2\,.$$
\bigskip

Dividing the identity of Corollary~\ref{thm.ztqq3ah} by $\mu^r$, and summing over $\mu$, we obtain
\begin{thm}
For $m\in\mathbb{Z^+}$, $r-2\in\mathbb{Z^+}$ and $n\in\mathbb{N}_0$ holds
\[
\begin{split}
&( - 1)^m \sum_{p = 0}^n {\binom{m+p-1}pT(r - m - 1,m + p,n - p + 1)}\\
&\qquad =  - E(n + 1,r - 1) + \sum_{i = 2}^m {( - 1)^i \binom{i+n-1}n\zeta (n + i)\zeta (r - i)}\,.
\end{split}
\]
\end{thm}

\begin{cor}
For $m,n\in\mathbb{Z^+}$ holds
\[
\begin{split}
&( - 1)^m 2\sum_{p = 0}^n {\binom{m+p-1}pT(n - m + 1,m + p,n - p + 1)}\\
&\qquad =  - \zeta(n+1)^2-\zeta(2n+2) + 2\sum_{i = 2}^m {( - 1)^i \binom{i+n-1}n\zeta (n + i)\zeta (r - i)}\,.
\end{split}
\]
\end{cor}

In particular,
\[
2\sum_{p = 1}^n {T(n-1,n - p + 1,p)}  = \zeta (n)^2  + \zeta (2n)\,,\quad n-1\in\mathbb{Z^+}\,.
\]

\bigskip

Dividing through the identity of Corollary~\ref{thm.cwe2ocv} by $\mu^r$ and summing over $\mu$ gives
\begin{thm}
For $r-1\in\mathbb{Z^+}$, $m,n\in\mathbb{N}_0$ and $r>m+1$ holds
\[
\begin{split}
&\sum_{p = 0}^n {\binom{m+p}mT(r - m - 1,n - p + 1,m + p + 1)}\\
&\quad\qquad = \sum_{i = n}^{m + n} {\binom inE(i + 1,r - i + n)}  + \sum_{i = n + 1}^{m + n} {\binom in\zeta (i + 1)\zeta (r - i + n)}\,.
\end{split}
\]

\end{thm}

\section{Evaluation of Tornheim double series}

\subsection{Euler-Zagier double zeta function}
Before discussing the general Tornheim double series for finite $r$, $s$ and $t$, we first consider the double series $T(r,0,t)$ and $T(0,s,t)$, with $r+t>2$ or $s+t>2$ and $t>0$.

\bigskip

According to~\eqref{equ.lf9ruko}, when $r$ and $t$ are positive integers greater than unity, then
\begin{equation}\label{equ.up0stgt}
T(r,0,t)=\zeta(r)\zeta(t)-E(t,r)\,.
\end{equation}
Thus, reduction of $T(r,0,t)$ to $\zeta$ values is possible when $r$ and $t$ are of different parity, in view of Theorem~3.1 of~\cite{flajolet}, and also when $r=t$ or $r+t=6$. 

\bigskip

Using, in~\eqref{equ.up0stgt}, the symmetry property of linear Euler sums, 
\begin{equation}\label{equ.csxx7lv}
E(m,n)+E(n,m)=\zeta(m+n)+\zeta(m)\zeta(n)\,,\quad\cite{flajolet, adegoke15}\,,
\end{equation}
we have
\begin{thm}
$$T(0,s,t)+T(0,t,s)=\zeta(s)\zeta(t)-\zeta(s+t)\,,\quad s-1,\,t-1\in\mathbb{Z^+}\,.$$
\end{thm}

\begin{cor}\label{cor.p5c00dq}
$$2\,T(0,s,s)=\zeta(s)^2-\zeta(2s)\,,\quad s-1\in\mathbb{Z^+}\,.$$
\end{cor}

\bigskip

$T(0,0,t)$ was evaluated, in~\cite{tornheim}, as
\begin{thm}
$$T(0,0,t)=\zeta(t-1)-\zeta(t)\,,\quad t>2\,.$$
\end{thm}
Here we give a different derivation as follows.
\begin{proof} From~\eqref{equ.i4135lr}, we have
\[
\begin{split}
\sum_{\mu  = 1}^N {\sum_{\nu  = 1}^N {\frac{1}{{(\nu  + \mu )^t }}} } &= \sum_{\mu  = 1}^N {\left( {H_{N + \mu ,t}  - H_{\mu ,t} } \right)}\\ 
&= \sum_{\mu  = 1}^N {H_{N + \mu ,t} }  - \sum_{\mu  = 1}^N {H_{\mu ,t} }\\ 
&= \sum_{\mu  = N + 1}^{2N} {H_{\mu ,t} }  - \sum_{\mu  = 1}^N {H_{\mu ,t} }\\ 
&= \sum_{\mu  = 1}^{2N} {H_{\mu ,t} }  - 2\sum_{\mu  = 1}^N {H_{\mu ,t} }\\
&= 2NH_{2N,t}  + H_{2N,t}  - H_{2N,t - 1}\\ 
&\qquad- 2NH_{N,t}  - 2H_{N,t}  + 2H_{N,t - 1}\,,
\end{split}
\]
and the result follows on taking limit $N\to\infty$.
\end{proof}
Note that in the final step of the above proof, we used $$\sum_{r=1}^NH_{r,n}=(N+1)H_{N,n}-NH_{N,n-1}\,,\quad\mbox{(identity~3.1 of \cite{adegoke15})}\,.$$

\subsection{Evaluation of the general Tornheim double series}

Dividing the identity of Corollary~\ref{thm.p1iwzvm} by $\mu^r$ and summing over $\mu$, we obtain
\begin{thm}\label{thm.bvgewqc} For $r\in\mathbb{N}_0$, $s,t\in\mathbb{Z^+}$, $\mbox{r+s>1}$ and $\mbox{r+t>1}$ holds
\[
\begin{split}
T(r,s,t) &= \sum_{i = 0}^{s - 2} {( - 1)^i \binom{t+i-1}i\zeta (s - i)\zeta (r + t + i)}\\
&\quad+ ( - 1)^s \sum_{i = 0}^{t - 2} {\binom{s+i-1}i\zeta (t - i)\zeta (r + s + i)}\\
&\qquad- ( - 1)^s \sum_{i = 0}^{t - 2} {\left\{ {\binom{s+i-1}i\sum_{\mu  = 1}^\infty  {\frac{{H_{\mu ,t - i} }}{{\mu ^{r + s + i} }}} } \right\}}\\
&\qquad\quad  - ( - 1)^s \binom{s+t-2}{t-1}\sum_{\mu  = 1}^\infty  {\frac{{H_\mu  }}{{\mu ^{r + s + t - 1} }}}\,.
\end{split}
\]
\end{thm}

\begin{cor}
\[
\begin{split}
T(r,s,1) &= \frac{{( - 1)^{s - 1} }}{2}\left[ {(r + s + 2)\zeta (r + s + 1) - \sum_{i = 1-r}^{s - 2} {\zeta (s - i)\zeta (r + i + 1)} } \right]\\
&\qquad+ \sum_{i = 0}^{s - 2} {( - 1)^i\zeta (s - i)\zeta (r + i + 1)}\,.
\end{split}
\]
In particular, we have the beautiful and well-known result

\[
T(1,1,1) = \sum_{\mu  = 1}^\infty  {\frac{{H_\mu  }}{{\mu ^2 }}}  = 2\zeta(3)\,.
\]

\end{cor}

We see immediately from Theorem~\ref{thm.bvgewqc} that due to the presence of the Euler sum \mbox{$E(t-i,i+r+s)$}, of weight \mbox{$w=r+s+t$}, complete reduction of $T(r,s,t)$ to $\zeta$ values is achieved, in general, if \mbox{$w$} is a positive odd integer or if $t=1$.

\bigskip

Using the index shift identity~\eqref{equ.mjf11t4}, the identity of Theorem~\ref{thm.bvgewqc} can also be written as
\begin{equation}\label{equ.llgretd}
\begin{split}
T(r,s,t) &= ( - 1)^t \sum_{i = t - s + 2}^t {( - 1)^i \binom {2t-i-1}{t-1}\zeta (s - t + i)\zeta (r + 2t - i)}\\ 
&\qquad + ( - 1)^s \sum_{i = 2}^t {\binom {s+t-i-1}{s-1}\zeta (i)\zeta (r + s + t - i)}\\ 
&\quad\qquad - ( - 1)^s \sum_{i = 1}^t {\binom {s+t-i-1}{s-1}\sum_{\mu  = 1}^\infty  {\frac{{H_{\mu ,i} }}{{\mu ^{r + s + t - i} }}} }\,,
\end{split}
\end{equation}
giving, in particular, for $s\in\mathbb{Z^+}$ and \mbox{$r+s>1$},
\begin{equation}\label{equ.n4vyrlv}
\begin{split}
(-1)^{s-1}T(r,s,s) &=-\sum_{i = 2}^s {\binom {2s-i-1}{s-1}\left((-1)^i+1\right)\zeta (i)\zeta (r + 2s - i)}\\ 
&\quad\qquad +\sum_{i = 1}^s {\binom {2s-i-1}{s-1}\sum_{\mu  = 1}^\infty  {\frac{{H_{\mu ,i} }}{{\mu ^{r + 2s - i} }}} }\\
&\\
&=-2\sum_{i = 1}^{\lfloor s/2\rfloor} {\binom {2s-2i-1}{s-1}\zeta (2i)\zeta (r + 2s - 2i)}\\ 
&\qquad\quad +\sum_{i = 1}^s {\binom {2s-i-1}{s-1}\sum_{\mu  = 1}^\infty  {\frac{{H_{\mu ,i} }}{{\mu ^{r + 2s - i} }}} }\,.
\end{split}
\end{equation}

\section{On linear Euler sums }
Explicit evaluations, in zeta values, are known for the sums,  
\[
\sum_{\nu  = 1}^\infty  {\frac{{H_{\nu ,n} }}{{\nu ^s }}}  + \sum_{\nu  = 1}^\infty  {\frac{{H_{\nu ,s} }}{{\nu ^n }}}\mbox{ and }\sum_{\nu  = 1}^\infty  {\frac{{H_{\nu} }}{{\nu ^s }}}\,,
\]
for $n-1,s-1\in\mathbb{Z^+}$, as expressed in identity~\eqref{equ.csxx7lv} and in the identity of Theorem~\ref{th.nj1j516}, respectively. Evaluation formulas are also known in the literature (see~\cite{sofo12}) for the following sums, for $\mu\in\mathbb{Z^+}$:
\[
\sum_{\nu  = 1}^\infty  {\frac{{H_\nu  }}{{(\mu  + \nu )^s }}} \mbox{ and }\sum_{\nu  = 1}^\infty  {\frac{{H_\nu  }}{{\nu (\mu  + \nu )}}}\,.
\]
It is our aim in this section to extend these results by deriving evaluation formulas for the following sums:
\[
\sum_{\nu  = 1}^\infty  {\frac{{H_{\nu ,n} }}{{(\nu  + \mu )^s }}}  + \sum_{\nu  = 1}^\infty  {\frac{{H_{\nu ,s} }}{{(\nu  + \mu )^n }}} \,,\quad\sum_{\nu  = 1}^\infty  {\frac{{H_\nu  }}{{(\nu  + \mu )^s }}}\,,\quad \sum_{\nu  = 1}^\infty  {\frac{{H_{\nu ,n} }}{{\nu (\nu  + \mu )}}}
\]
and
\[
\sum_{\nu  = 1}^\infty  {\frac{{H_{\nu} }}{{\nu^s (\nu  + \mu )^t}}}\,.
\]
We will also derive variants of the Euler formula for $E(1,n)$ and obtain certain combinations of linear Euler sums that evaluate to zeta values.

\subsection{Extension of known results for Euler sums}
Consider the double sum
\[
f(\mu ,s,n) = \sum_{\nu  = 1}^\infty  {\sum_{i = 1}^{\mu  - 1} {\frac{1}{{(\nu  + \mu )^s (\nu  + i )^n }}} }\,.
\]
A change of the order of summation and shifting of the indices of summation give
\begin{equation}\label{equ.tg63b4s}
f(\mu ,s,n) = \sum_{i = 1}^{\mu  - 1} {\sum_{\nu  = 1}^\infty  {\frac{1}{{\nu ^n (\nu  + i )^s }}} }  - \sum_{i = 1}^{\mu  - 1} {\sum_{\nu  = 1}^{\mu  - i} {\frac{1}{{\nu ^n (\nu  + i )^s }}} }\,.
\end{equation}
But
\[
\begin{split}
f(\mu ,s,n) &= \sum_{\nu  = 1}^\infty  {\sum_{i = 1}^{\mu  - 1} {\frac{1}{{(\nu  + \mu )^s (\nu  + i )^n }}} }  = \sum_{\nu  = 1}^\infty  {\left\{ {\frac{1}{{(\nu  + \mu )^s }}\sum_{i = 1}^{\mu  - 1} {\frac{1}{{(\nu  + i )^n }}} } \right\}}\\
&=\sum_{\nu  = 1}^\infty  {\frac{{H_{\nu  + \mu  - 1,n} }}{{(\nu  + \mu )^s }}}  - \sum_{\nu  = 1}^\infty  {\frac{{H_{\nu ,n} }}{{(\nu  + \mu )^s }}}\,,\quad\mbox{ by identity~\eqref{equ.i4135lr} }\,,
\end{split}
\]
so that after index shifting and the use of identity \mbox{$H_{r-1,n}=H_{r,n}-1/r^n$}, we have
\begin{equation}\label{equ.c7iaghg}
f(\mu ,s,n) = \sum_{\nu  = 1}^\infty  {\frac{{H_{\nu ,n} }}{{\nu ^s }}}  - \sum_{\nu  = 1}^\infty  {\frac{{H_{\nu ,n} }}{{(\nu  + \mu )^s }}}  - \sum_{\nu  = 1}^\mu  {\frac{{H_{\nu ,n} }}{{\nu ^s }}}  - \zeta (n + s)+H_{\mu,n+s}\,.
\end{equation}
On equating~\eqref{equ.tg63b4s} and \eqref{equ.c7iaghg} we obtain the identity
\begin{equation}\label{equ.spedt89}
\begin{split}
\sum_{\nu  = 1}^\infty  {\frac{{H_{\nu ,n} }}{{\nu ^s }}}  - \sum_{\nu  = 1}^\infty  {\frac{{H_{\nu ,n} }}{{(\nu  + \mu )^s }}}  &= \sum_{i = 1}^{\mu  - 1} {\sum_{\nu  = 1}^\infty  {\frac{1}{{\nu ^n (\nu  + i )^s }}} }  - \sum_{i = 1}^{\mu  - 1} {\sum_{\nu  = 1}^{\mu  - i} {\frac{1}{{\nu ^n (\nu  + i )^s }}} }\\
&\qquad\qquad+ \sum_{\nu  = 1}^\mu  {\frac{{H_{\nu ,n} }}{{\nu ^s }}}  + \zeta (n + s) - H_{\mu,n + s}\,,
\end{split}
\end{equation}
which holds for \mbox{$n,s-1\in\mathbb{Z^+}$} and \mbox{$\mu\in\mathbb{N}_0$}, from which we will derive a couple of interesting results. First we state a lemma.

\begin{lemma}\label{thm.eclzyl6}
For $\mu\in\mathbb{N}_0$ and $s,n\in\mathbb{Z}$ such that either $sn\in\mathbb{Z^-}$ or $s+n\in\mathbb{Z^+}$ holds
\[
\begin{split}
\sum_{\nu  = 1}^\mu  {\frac{{H_{\mu  - \nu ,s} }}{{\nu ^n }}}  &= \sum_{j = 1}^n {\binom{n+s-j-1}{s-1}\sum_{\nu  = 1}^\mu  {\frac{{H_{\nu  - 1,j} }}{{\nu ^{n + s - j} }}} }\\
&\qquad+ \sum_{j = 1}^s {\binom{n+s-j-1}{n-1}\sum_{\nu  = 1}^\mu  {\frac{{H_{\nu  - 1,j} }}{{\nu ^{n + s - j} }}} }\\
&\\
&= \sum_{j = 1}^n {\binom{n+s-j-1}{s-1}\sum_{\nu  = 1}^\mu  {\frac{{H_{\nu,j} }}{{\nu ^{n + s - j} }}} }\\
&\qquad+ \sum_{j = 1}^s {\binom{n+s-j-1}{n-1}\sum_{\nu  = 1}^\mu  {\frac{{H_{\nu,j} }}{{\nu ^{n + s - j} }}} }\\
&\quad\qquad-\binom{n+s}nH_{\mu,n+s}\,.
\end{split}
\]
In particular,
\begin{equation}\label{equ.srgz6mr}
\begin{split}
\sum_{\nu  = 1}^\mu  {\frac{{H_{\mu  - \nu } }}{{\nu ^n }}}&= \sum_{j = 1}^n {\sum_{\nu  = 1}^\mu  {\frac{{H_{\nu  - 1,j} }}{{\nu ^{n - j + 1} }}} }  + \sum_{\nu  = 1}^\mu  {\frac{{H_{\nu  - 1} }}{{\nu ^n }}}\\  
&= \sum_{j = 2}^n {\sum_{\nu  = 1}^\mu  {\frac{{H_{\nu ,j} }}{{\nu ^{n - j + 1} }}} }  + 2\sum_{\nu  = 1}^\mu  {\frac{{H_\nu  }}{{\nu ^n }}}  - (n + 1)H_{\mu ,n + 1}
\end{split}
\end{equation}
and since
\[
2\sum_{\nu  = 1}^\mu {\frac{{H_{\nu ,n} }}{{\nu ^n }}}  = H_{\mu,n} ^2  + H_{\mu,2n}\,,\quad\mbox{(identity (3.25) of \cite{adegoke15}) }\,, 
\]
we have
\begin{equation}\label{equ.lp5bakz}
\sum_{\nu  = 1}^\mu  {\frac{{H_{\mu  - \nu } }}{\nu }}  = 2\sum_{\nu  = 1}^\mu  {\frac{{H_{\nu  - 1} }}{\nu }}  = H_\mu ^2  - H_{\mu ,2}\,.
\end{equation}
\end{lemma}

\begin{proof}

For $\nu\in\mathbb{N}_0$ and $s,n\in\mathbb{Z}$, define
$$K(\nu;n,s)=\sum_{i=1}^\nu{\frac{H_{\nu-i,s}}{i^n}}\,,$$ and note that \mbox{$K(0;n,s)=0=K(1;n,s)$}. 
With this definition and the identity~\eqref{equ.i4135lr} in mind, we obtain the recursion
\[
K(\nu ;n,s) - K(\nu  - 1;n,s) = \sum_{i = 1}^{\nu  - 1} {\frac{1}{{i^n (\nu  - i)^s }}}\,, 
\]
from which the result follows by invoking Theorem~\ref{thm.u9a6spe} to resolve the right hand side and then summing both sides from $\nu=1$ to $\nu=\mu$, noting that the sum on the left hand side telescopes.

\end{proof}
\bigskip
From the identity of Lemma~\ref{thm.eclzyl6} we note the symmetry property \mbox{$K(\nu ;n,s)=K(\nu ;s,n)$}.
\subsubsection{Extension of the Euler formula of Theorem~\ref{th.nj1j516}}
Setting $n=1$ in the identity~\eqref{equ.spedt89}, employing the identity of Theorem~\ref{th.nj1j516} and using the identity~\eqref{equ.gqhxhjb} to simplify the double sums, we obtain
\[
\begin{split}
2\sum_{\nu  = 1}^\infty  {\frac{{H_\nu  }}{{(\nu  + \mu )^s }}}  &= s\,\zeta (s + 1) - \sum_{i = 1}^{s - 2} {\zeta (i + 1)\zeta (s - i)}\\
&\qquad+ 2\sum_{i = 1}^{s - 1} {H_{\mu  - 1,s - i} \left( {\zeta (i + 1) - H_{\mu ,i + 1} } \right)}\\
&\quad\qquad+ 2\sum_{i = 1}^\mu  {i^{ - s} \left( {H_{\mu  - i}  - H_i } \right)}\\
&\qquad\qquad  + 2H_{\mu ,s + 1}  - 2H_{\mu  - 1,s} H_\mu\,,
\end{split}
\]
which in view of identity~\eqref{equ.srgz6mr} now gives, after some algebra,
\begin{thm}\label{thm.ryy2yk3}
For $\mu,s-1\in\mathbb{Z^+}$ holds
\[
\begin{split}
2\sum_{\nu  = 1}^\infty  {\frac{{H_\nu  }}{{(\nu  + \mu )^s }}} &= 2H_{\mu  - 1} \left( {\zeta (s) - H_{\mu  - 1,s} } \right)\\
&\quad + s\left( {\zeta (s + 1) - H_{\mu  - 1,s + 1} } \right)\\
&\qquad- \sum_{i = 1}^{s - 2} {\left\{ {\left( {\zeta (i + 1) - H_{\mu  - 1,i + 1} } \right)\left( {\zeta (s - i) - H_{\mu  - 1,s - i} } \right)} \right\}}\,.
\end{split}
\]

In particular, $$2\sum_{\nu  = 1}^\infty  {\frac{{H_\nu  }}{{(\nu  + 1 )^s }}}  = s\,\zeta (s + 1) - \sum_{j = 1}^{s - 2} {\zeta (j + 1)\zeta (s - j)}\,,\quad s-1\in\mathbb{Z^+}\,.$$
\end{thm}

Using a different formalism, the identity of Theorem~\ref{thm.ryy2yk3} was first derived in~\cite{sofo12}.
\subsubsection{Extension of the symmetry relation for linear Euler sums}
From the identity~\eqref{equ.spedt89} and the symmetry relations
\[
\sum_{\nu  = 1}^\infty  {\frac{{H_{\nu ,n} }}{{\nu ^s }}}  + \sum_{\nu  = 1}^\infty  {\frac{{H_{\nu ,s} }}{{\nu ^n }}}  = \zeta (n + s) + \zeta (n)\zeta (s)
\]
and
\[
\sum_{\nu  = 1}^\mu  {\frac{{H_{\nu ,n} }}{{\nu ^s }}}  + \sum_{\nu  = 1}^\mu  {\frac{{H_{\nu ,s} }}{{\nu ^n }}}  = H_{\mu ,n + s}  + H_{\mu ,n} H_{\mu ,s}\,,\quad \mbox{(equation~(3.22) of~\cite{adegoke15})}\,, 
\]
we have
\begin{equation}\label{equ.vdgftrd}
\begin{split}
&\sum_{\nu  = 1}^\infty  {\frac{{H_{\nu ,n} }}{{(\nu  + \mu )^s }}}  + \sum_{\nu = 1}^\infty  {\frac{{H_{\nu ,s} }}{{(\nu  + \mu )^n }}} \\
&\qquad= H_{\mu ,n + s}  - H_{\mu ,n} H_{\mu ,s}\\ 
&\quad\qquad - \zeta (n + s) + \zeta (n)\zeta (s)\\
&\qquad\qquad - \sum_{i = 1}^{\mu  - 1} {\sum_{\nu  = 1}^\infty  {\frac{1}{{\nu ^n (\nu  + i )^s }}} }  - \sum_{i = 1}^{\mu  - 1} {\sum_{\nu  = 1}^\infty  {\frac{1}{{\nu ^s (\nu  + i )^n }}} }\\ 
&\quad\qquad\qquad+ \sum_{i = 1}^{\mu  - 1} {\sum_{\nu  = 1}^{\mu  - i} {\frac{1}{{\nu ^n (\nu  + i )^s }}} }  + \sum_{i = 1}^{\mu  - 1} {\sum_{\nu  = 1}^{\mu  - i} {\frac{1}{{\nu ^s (\nu  + i )^n }}} }\,. 
\end{split}
\end{equation}
Using the identity of Corollary~\ref{thm.p1iwzvm} we establish that
\begin{equation}\label{equ.hfmh2se}
\begin{split}
\sum_{i = 1}^{\mu  - 1} {\sum_{\nu  = 1}^\infty  {\frac{1}{{\nu ^n (\nu  + i )^s }}} } & = \sum_{j = 0}^{n - 2} {\binom{s+j-1}j( - 1)^j \zeta (n - j)H_{\mu  - 1,s + j} }\\
&\quad+ ( - 1)^n \sum_{j = 0}^{s - 2} {\binom{n+j-1}j\zeta (s - j)H_{\mu  - 1,n + j} }\\
&\qquad+ ( - 1)^{n - 1} \sum_{j = 0}^{s - 1} {\left\{ {\binom{n+j-1}j\sum_{i = 1}^{\mu  - 1} {\frac{{H_{i,s - j} }}{{i^{n + j} }}} } \right\}}\,.
\end{split}
\end{equation}
From the identity of Theorem~\ref{thm.b2i7lfz} and the identity of Lemma~\ref{thm.eclzyl6} follows that, for $n+s>0$, $n,s\in\mathbb{Z}$ and $\mu\in\mathbb{Z^+}$,
\begin{equation}\label{equ.te7cg80}
\begin{split}
\sum_{i = 1}^{\mu  - 1} {\sum_{\nu  = 1}^{\mu  - i} {\frac{1}{{\nu ^n (\nu  + i )^s }}} }  &= \sum_{j = 0}^{n - 1} {\left\{ {\binom{s+j-1}j( - 1)^j \sum_{k = 1}^{s + j} {\left[ {\binom{n+s-k-1}{n-j-1}\sum_{i = 1}^\mu  {\frac{{H_{i - 1,k} }}{{i^{n + s - k} }}} } \right]} } \right\}}\\ 
&\quad + \sum_{j = 0}^{n - 1} {\left\{ {\binom{s+j-1}j( - 1)^j \sum_{k = 1}^{n - j} {\left[ {\binom{n+s-k-1}{s+j-1}\sum_{i = 1}^\mu  {\frac{{H_{i - 1,k} }}{{i^{n + s - k} }}} } \right]} } \right\}}\\ 
&\qquad + ( - 1)^n \sum_{j = 0}^{s - 1} {\binom{n+j-1}jH_{\mu ,s - j} H_{\mu  - 1,n + j} }\\ 
&\quad\qquad + ( - 1)^{n - 1} \sum_{j = 0}^{s - 1} {\left\{ {\binom{n+j-1}j\sum_{i = 1}^{\mu  - 1} {\frac{{H_{i,s - j} }}{{i^{n + j} }}} } \right\}}\,.
\end{split}
\end{equation}
Plugging~\eqref{equ.hfmh2se} and \eqref{equ.te7cg80} in~\eqref{equ.vdgftrd} we prove
\begin{thm}\label{thm.lv0bcn0}
For $\mu$, $n-1$ and $s-1\in\mathbb{Z^+}$ holds
\[
\begin{split}
&\sum_{\nu  = 1}^\infty  {\frac{{H_{\nu ,n} }}{{(\nu  + \mu )^s }}}  + \sum_{\nu  = 1}^\infty  {\frac{{H_{\nu ,s} }}{{(\nu  + \mu )^n }}}\\
&\quad= H_{\mu ,n + s}  - H_{\mu ,n} H_{\mu ,s}- \zeta (n + s) + \zeta (n)\zeta (s)\\
&\qquad- \sum_{j = 0}^{n - 2} {\binom{s+j-1}j\left[ {( - 1)^j  + ( - 1)^s } \right]\zeta (n - j)H_{\mu  - 1,s + j} }\\
&\quad\qquad - \sum_{j = 0}^{s - 2} {\binom{n+j-1}j\left[ {( - 1)^j  + ( - 1)^n } \right]\zeta (s - j)H_{\mu  - 1,n + j} }\\
&\quad+ ( - 1)^n \sum_{j = 0}^{s - 1} {\binom{n+j-1}jH_{\mu ,s - j} H_{\mu  - 1,n + j} }+ ( - 1)^s \sum_{j = 0}^{n - 1} {\binom{s+j-1}jH_{\mu ,n - j} H_{\mu  - 1,s + j} }\\
&\quad+ \sum_{j = 0}^{n - 1} {\left\{ {\binom{s+j-1}j( - 1)^j \sum_{k = 1}^{s + j} {\left[ {\binom{n+s-k-1}{n-j-1}\sum_{i = 1}^\mu  {\frac{{H_{i - 1,k} }}{{i^{n + s - k} }}} } \right]} } \right\}}\\
&\quad\qquad+ \sum_{j = 0}^{s - 1} {\left\{ {\binom{n+j-1}j( - 1)^j \sum_{k = 1}^{n + j} {\left[ {\binom{n+s-k-1}{s-j-1}\sum_{i = 1}^\mu  {\frac{{H_{i - 1,k} }}{{i^{n + s - k} }}} } \right]} } \right\}}\\
&\qquad\qquad+ \sum_{j = 0}^{n - 1} {\left\{ {\binom{s+j-1}j( - 1)^j \sum_{k = 1}^{n - j} {\left[ {\binom{n+s-k-1}{s+j-1}\sum_{i = 1}^\mu  {\frac{{H_{i - 1,k} }}{{i^{n + s - k} }}} } \right]} } \right\}}\\
&\quad\qquad+ \sum_{j = 0}^{s - 1} {\left\{ {\binom{n+j-1}j( - 1)^j \sum_{k = 1}^{s - j} {\left[ {\binom{n+s-k-1}{n+j-1}\sum_{i = 1}^\mu  {\frac{{H_{i - 1,k} }}{{i^{n + s - k} }}} } \right]} } \right\}}\,.
\end{split}
\]
In particular
\[
\sum_{\nu  = 1}^\infty  {\frac{{H_{\nu ,n} }}{{(\nu  + 1 )^s }}}  + \sum_{\nu  = 1}^\infty  {\frac{{H_{\nu ,s} }}{{(\nu  + 1 )^n }}}=\zeta (n)\zeta (s) - \zeta (n + s)
\]
and
\[
\begin{split}
&2\sum_{\nu  = 1}^\infty  {\frac{{H_{\nu ,n} }}{{(\nu  + \mu )^n }}}  = H_{\mu ,2n}  - \zeta (2n) - H_{\mu ,n}^2  + \zeta (n)^2\\ 
&\qquad\quad + ( - 1)^{n - 1} 4\sum_{j = 1}^{\left\lfloor {n/2} \right\rfloor } {\binom{2n-2j-1}{n-1}\zeta (2j)H_{\mu  - 1,2n - 2j} }\\
&\qquad - ( - 1)^{n - 1} 2\sum_{j = 1}^n {\binom{2n-j-1}{n-1}( - 1)^j H_{\mu  - 1,2n - j} H_{\mu ,j} }\\ 
&\quad\qquad - ( - 1)^{n - 1} 2\sum_{j = 1}^n {\left\{ {\binom{2n-j-1}{n-1}( - 1)^j \sum_{k = 1}^{2n - j} {\left[ {\binom{2n-k-1}{j-1}\sum_{i = 1}^\mu  {\frac{{H_{i - 1,k} }}{{i^{2n - k} }}} } \right]} } \right\}}\\ 
&\qquad\qquad - ( - 1)^{n - 1} 2\sum_{j = 1}^n {\left\{ {\binom{2n-j-1}{n-1}( - 1)^j \sum_{k = 1}^j {\left[ {\binom{2n-k-1}{2n-j-1}\sum_{i = 1}^\mu  {\frac{{H_{i - 1,k} }}{{i^{2n - k} }}} } \right]} } \right\}}\,. 
\end{split}
\]

\end{thm}

\subsubsection{Evaluation of a certain type of Euler sum}
Setting $s=1$ in the identity~\eqref{equ.spedt89}, employing the identity~\eqref{equ.k9w7fxv} and the identity resulting from setting $n=0$ in the identity of Theorem~\ref{thm.x1nstmi} to simplify the double sums, we obtain
\[
\begin{split}
\mu \sum_{\nu  = 1}^\infty  {\frac{{H_{\nu ,n} }}{{\nu (\nu  + \mu )}}} &= ( - 1)^{n - 1} \sum_{j = 1}^{n - 1} {( - 1)^j H_{\mu  - 1,n - j}\, \zeta (j + 1)}\\
&\quad+ ( - 1)^{n - 1} \sum_{i = 1}^{\mu  - 1} {\sum_{k = 1}^n {( - 1)^k i^{k - n - 1} H_{\mu  - i,k} } }\\
&\qquad  + \left( {1 + ( - 1)^{n - 1} } \right)H_\mu  H_{\mu  - 1,n}\\
&\quad\qquad- \sum_{i = 1}^{\mu  - 1} {i^{ - n} H_i }  + \zeta (n + 1)\,,
\end{split}
\]
which, on account of Lemma~\ref{thm.eclzyl6} yields the evaluation:
\begin{thm}\label{thm.x9at23s} For $\mu,n\in\mathbb{Z^+}$ holds
\[
\begin{split}
\mu \sum_{\nu  = 1}^\infty  {\frac{{H_{\nu ,n} }}{{\nu (\nu  + \mu )}}} &= ( - 1)^{n - 1} \sum_{j = 1}^{n - 1} {( - 1)^j H_{\mu  - 1,n - j}\, \zeta (j + 1)}\\ 
&\quad + ( - 1)^{n - 1} \sum_{k = 1}^n \left\{( - 1)^k{\sum_{j = 1}^{n - k + 1} \left[{ { \binom{n-j}{k-1}\sum_{i = 1}^{\mu} {\frac{{H_{i - 1,j} }}{{i^{n - j + 1} }}} }} \right]}\right\}\\ 
&\qquad + ( - 1)^{n - 1} \sum_{k = 1}^n\left\{ ( - 1)^k{\sum_{j = 1}^k\left[ {{\binom{n-j}{n-k}\sum_{i = 1}^{\mu} {\frac{{H_{i - 1,j} }}{{i^{n - j + 1} }}} }} \right]}\right\}\\ 
&\quad\qquad + \left( {1 + ( - 1)^{n - 1} } \right)H_\mu  H_{\mu  - 1,n}\\
&\qquad\qquad - \sum_{i = 1}^{\mu  - 1} {\frac{H_i}{i^n} }  + \zeta (n + 1)\,.
\end{split}
\]
In particular, setting $n=1$ and using~\eqref{equ.lp5bakz} gives
\[
2\mu \sum_{\nu  = 1}^\infty  {\frac{{H_\nu  }}{{\nu (\nu  + \mu )}}}  = H_{\mu  - 1}^2  + H_{\mu  - 1,2}  + 2\zeta (2)\,,
\]
which (but not the general Theorem~\ref{thm.x9at23s}) was also derived in~\cite{sofo12}; while setting $\mu=1$ in the theorem gives
\[
\sum_{\nu  = 1}^\infty  {\frac{{H_{\nu ,n} }}{{\nu (\nu + 1 )}}}  = \zeta (n + 1)\,,
\]
which is a more general case of the result reported in~\cite{chu12} for $n=1$.

\end{thm}

If we multiply through~\eqref{equ.p5bh6gf} by $H_\nu$ and sum over $\nu$ while making use of the identities of Theorems~\ref{th.nj1j516}, \ref{thm.ryy2yk3} and \ref{thm.x9at23s}, we obtain
\begin{thm}
For $\mu,s,t\in\mathbb{Z^+}$ holds
\[
\begin{split}
&2\sum_{\nu  = 1}^\infty  {\frac{{H_\nu  }}{{\nu ^s (\nu  + \mu )^t }}}  = ( - 1)^{s - 1} \binom{s+t-2}{s-1}\frac{{H_{\mu  - 1}^2  + H_{\mu  - 1,2}  + 2\zeta (2)}}{{\mu ^{s + t - 1} }}\\
&\qquad\qquad + \sum_{i = 0}^{s - 2} {\binom{t+i-1}i\frac{{( - 1)^i (s - i + 2)\zeta (s - i + 1)}}{{\mu ^{t + i} }}}\\ 
&\qquad\qquad\quad - \sum_{i = 0}^{s - 2} {\left\{ {\binom{t+i-1}i\frac{{( - 1)^i }}{{\mu ^{t + i} }}\sum_{j = 1}^{s - i - 2} {\zeta (j + 1)\zeta (s - i - j)} } \right\}}\\ 
&\qquad\quad + 2( - 1)^s H_{\mu  - 1} \sum_{i = 0}^{t - 2} {\binom{s+i-1}i\frac{{\left( {\zeta (t - i) - H_{\mu  - 1,t - i} } \right)}}{{\mu ^{s + i} }}}\\
&\quad + ( - 1)^s \sum_{i = 0}^{t - 2} {\binom{s+i-1}i\frac{{(t - i)\left( {\zeta (t - i + 1) - H_{\mu  - 1,t - i + 1} } \right)}}{{\mu ^{s + i} }}}\\
&\qquad - ( - 1)^s \sum_{i = 0}^{t - 2} {\left\{ {\binom{s+i-1}i\frac{1}{{\mu ^{s + i} }}\sum_{j = 1}^{t - i - 2} {\left( {\zeta (j + 1) - H_{\mu  - 1,j + 1} } \right)\left( {\zeta (t - i - j) - H_{\mu  - 1,t - i - j} } \right)} } \right\}}\,. 
\end{split}
\]

\end{thm}

\subsection{Variants of Euler formula for $E(1,2s+1)$ and certain combinations of linear Euler sums that evaluate to zeta values}
Using the reflection symmetry~\eqref{equ.qg21457} of the Tornheim double series in the identity of Theorem~\ref{thm.bvgewqc} and setting \mbox{$r=s+1$} to ensure that $r$ and $s$ have different parity, we obtain
\begin{equation}\label{equ.ylv3nb4}
\begin{split}
&\sum_{i = 1}^{t - 1} {\left[ {\binom {i+s-1}s\frac{{2s + i - 1}}{{i + s - 1}}\sum_{\mu  = 1}^\infty  {\frac{{H_{\mu ,t - i + 1} }}{{\mu ^{2s + i} }}} } \right]}\\
&\quad  + \frac{{2s + t - 1}}{{s + t - 1}}\binom {s+t-1}{t-1}\sum_{\mu  = 1}^\infty  {\frac{{H_\mu  }}{{\mu ^{2s + t} }}}\\
&\qquad= \sum_{i = 1}^{t - 1} {\binom {i+s-1}s\frac{{2s + i - 1}}{{i + s - 1}}\zeta (t - i + 1)\zeta (2s + i)}\\
&\quad\qquad+ ( - 1)^{s + 1} \sum_{i = 0}^{s - 1} {( - 1)^i \binom {i+t-1}{t-1}\frac{{2i + t - 1}}{{i + t - 1}}\zeta (s - i + 1)\zeta (s + i + t)}\,.
\end{split}
\end{equation}

\subsubsection{Variants of Euler formula for $E(1,2s+1)$}
On setting $t=1$ in~\eqref{equ.ylv3nb4}, we obtain
\begin{thm}\label{thm.ctfkdby}
For $s\in\mathbb{Z^+} holds$
\[
2( - 1)^{s + 1}\sum_{\mu  = 1}^\infty  {\frac{{H_\mu  }}{{\mu ^{2s + 1} }}}  =  \zeta (s + 1)^2  + 2\sum_{i = 1}^{s - 1} {( - 1)^i \zeta (s - i + 1)\zeta (s + i + 1)}\,, 
\]
or equivalently, using the index shift identity,
\begin{equation}\label{equ.nvv0110}
2\sum_{\mu  = 1}^\infty  {\frac{{H_\mu  }}{{\mu ^{2s + 1} }}}  = ( - 1)^{s - 1} \zeta (s + 1)^2  + 2\sum_{i = 0}^{s - 2} {( - 1)^i \zeta (i + 2)\zeta (2s - i)}\,.
\end{equation}

\end{thm}

Dividing through~\eqref{equ.k9w7fxv} by $\mu^{m+1}$ and summing over $\mu$ gives
\[
( - 1)^{m - 1} \sum_{\nu  = 1}^\infty  {\left\{\frac{1}{{\nu ^m }}\sum_{\mu  = 1}^\infty  {\frac{{H_\mu  }}{{\mu (\nu  + \mu )}}} \right\}}  = \sum_{\mu  = 1}^\infty  {\frac{{H_\mu  }}{{\mu ^{m + 1} }}}  + \sum_{i = 1}^{m - 1} {( - 1)^i \zeta (i + 1)\zeta (m - i + 1)}\,, 
\]
that is
\[
( - 1)^{m - 1} \sum_{\nu  = 1}^\infty  {\frac{{H_\nu  }}{{\nu ^{m + 1} }}}  = \sum_{\nu  = 1}^\infty  {\frac{{H_\nu  }}{{\nu ^{m + 1} }}}  + \sum_{i = 1}^{m - 1} {( - 1)^i \zeta (i + 1)\zeta (m - i + 1)}\,, 
\]
from which, by setting $m=2s$ and shifting the summation index in the second sum of the right hand side, we obtain
\begin{thm}
For $s\in\mathbb{Z^+}$ holds
\begin{equation}\label{equ.lm6twhv}
2\sum_{\nu  = 1}^\infty  {\frac{{H_\nu  }}{{\nu ^{2s + 1} }}}  = \sum_{i = 0}^{2s - 2} {( - 1)^i \zeta (i + 2)\zeta (2s - i)}\,, 
\end{equation}
a result that was derived first in~\cite{georghiou} and later in~\cite{boyad}.

\end{thm}
Note that the identities~\eqref{equ.nvv0110} and \eqref{equ.lm6twhv} are equivalent since
\[
\begin{split}
&( - 1)^{s - 1} \zeta (s + 1)^2  + 2\sum_{i = 0}^{s - 2} {( - 1)^i \zeta (i + 2)\zeta (2s - i)}\\
&\qquad=\sum_{i = 0}^{s - 2} {( - 1)^i \zeta (i + 2)\zeta (2s - i)}\\
&\quad\qquad+\sum_{i = 0}^{s - 1} {( - 1)^i \zeta (i + 2)\zeta (2s - i)}
\end{split}
\]
which is equivalent to the sum on right side of~\eqref{equ.lm6twhv}.

\bigskip

Setting $n=2s+1$ in the identity of Theorem~\ref{th.nj1j516} and eliminating $E(1,2s+1)$ between the resulting identity and~\eqref{equ.lm6twhv}, we obtain, after a little manipulation,
\begin{thm}[Euler]
\[(2s+1)\zeta(2s)=2\sum_{i=1}^{s-1}{\zeta(2s-2i)\zeta(2i)}\,,\quad s-1\in\mathbb{Z^+}\,,\]
\end{thm}
the well-known zeta function relation due to Euler.

\subsubsection{Certain combinations of linear Euler sums that \mbox{evaluate} to zeta values}

Researchers have noted that linear Euler sums of even weight are probably not reducible to zeta values alone~\cite{georghiou, borwein295, flajolet}. In a 1998 paper~\cite{flajolet}, Flajolet and Salvy gave a couple of examples of linear combinations of Euler sums of even weight, expressed in terms of the Riemann zeta function. Such evaluations are also found in~\cite{huard,espinosa2,espinosa1}. In this section we discover certain combinations of linear Euler sums that evaluate to zeta values. 

\bigskip

Tornheim proved that (equation (8) of~\cite{tornheim})
\[
T(1,1,s) = (s + 1)\zeta (s + 2) - \sum_{i = 2}^s {\zeta (i)\zeta (s - i + 2)},\quad s\in\mathbb{Z^+}\,, 
\]
which, in view of~\eqref{equ.llgretd}, gives
\[
\sum_{i = 1}^s {\left[ {\sum_{\mu  = 1}^\infty  {\frac{{H_{\mu ,i} }}{{\mu ^{s - i + 2} }}} } \right]}  = \sum_{i = 1}^s {E(i,s - i + 2)}  = (s + 1)\zeta (s + 2)\,,
\]
from which it follows that
\begin{thm} For $s\in\mathbb{Z^+}$ holds
\[
2\sum_{i = 2}^s {E(i,s - i + 2)}  = (s - 1)\zeta (s + 2)+\sum_{j=1}^{s-1}{\zeta(j+1)\zeta(s-j+1)}\,.
\]

\end{thm}

\bigskip

Setting $r=0$ in~\eqref{equ.n4vyrlv} and using Corollary~\ref{cor.p5c00dq}, we obtain
\[
\begin{split}
&2\sum_{i = 1}^s {\binom {2s-i-1}{s-1}\sum_{\mu  = 1}^\infty  {\frac{{H_{\mu ,i} }}{{\mu ^{2s - i} }}} }\\
&\quad =4\sum_{i = 1}^{\lfloor s/2\rfloor} {\binom {2s-2i-1}{s-1}\zeta (2i)\zeta (2s - 2i)}\\
&\quad\qquad +(-1)^{s-1}\left(\zeta(s)^2-\zeta(2s)\right)\,,
\end{split}
\]
from which we get
\begin{thm}\label{thm.xzc9oni} For $s-1\in\mathbb{Z^+}$ holds
\[
\begin{split}
&2\sum_{i = 2}^{s-1} {\binom {2s-i-1}{s-1}\sum_{\mu  = 1}^\infty  {\frac{{H_{\mu ,i} }}{{\mu ^{2s - i} }}} }\\
&\quad =4\sum_{i = 1}^{\lfloor s/2\rfloor} {\binom {2s-2i-1}{s-1}\zeta (2i)\zeta (2s - 2i)}\\
&\qquad -2\binom {2s-2}{s-1}\sum_{i=2}^s{(-1)^i\zeta(i)\zeta(2s-i)}\\
&\quad\qquad -(-1)^{s-1}\zeta(s)^2\left[\binom {2s-2}{s-1}+(-1)^{s-1}-1\right]\\
&\qquad\qquad -(-1)^{s-1}\zeta(2s)\left[(-1)^{s-1}+1\right]\,,
\end{split}
\]
after using Theorem~\ref{th.nj1j516} to write $E(1,2s-1)$ and using also the fact that \mbox{$2E(s,s)=\zeta(s)^2+\zeta(2s)$}.
\end{thm}
If we set $s=3$ in the above identity we obtain the well known result
\begin{cor}\label{thm.t7si22d}
\[
3\sum_{\mu=1}^\infty{\frac{H_{\mu,2}}{\mu^4}}=3\,\zeta(3)^2-\zeta(6)\,.
\]
\end{cor}
At $s=5$ and $s=6$, respectively, we have
\begin{cor}\label{thm.ab4k60c}
\[
\begin{split}
&7\sum_{\mu  = 1}^\infty  {\frac{{H_{\mu ,2} }}{{\mu ^8 }}}  + 3\sum_{\mu  = 1}^\infty  {\frac{{H_{\mu ,3} }}{{\mu ^7 }}}  + \sum_{\mu  = 1}^\infty  {\frac{{H_{\mu ,4} }}{{\mu ^6 }}}\\ 
&\qquad\qquad =  - 12\zeta (4)\zeta (6) + 14\zeta (3)\zeta (7) + 7\zeta (5)^2  - \frac{{\zeta (10)}}{5}
\end{split}
\]
\end{cor}
and
\begin{cor}
\[
\begin{split}
&126\sum_{\mu  = 1}^\infty  {\frac{{H_{\mu ,2} }}{{\mu ^{10} }}}  + 56\sum_{\mu  = 1}^\infty  {\frac{{H_{\mu ,3} }}{{\mu ^9 }}}  + 21\sum_{\mu  = 1}^\infty  {\frac{{H_{\mu ,4} }}{{\mu ^8 }}}  + 6\sum_{\mu  = 1}^\infty  {\frac{{H_{\mu ,5} }}{{\mu ^7 }}}\\ 
&\qquad\qquad =  - 210\zeta (4)\zeta (8) - 125\zeta (6)^2  + 252\zeta (3)\zeta (9) + 252\zeta (5)\zeta (7)\,.
\end{split}
\]
\end{cor}

\bigskip

It was proved in~\cite{huard} that
\begin{equation}\label{equ.m2uzmco}
T(s,s,s) = \frac{4}{{1 + 2( - 1)^s }}\sum_{i = 0}^{\lfloor {s/2} \rfloor } {\binom{2s-2i-1}{s-1}\zeta (2i)\zeta (3s - 2i)}\,.
\end{equation}
Setting $r=s$ in~\eqref{equ.n4vyrlv} and equating with~\eqref{equ.m2uzmco} gives
\[
\begin{split}
&\sum_{i = 1}^s {\binom{2s-i-1}{s-1}\sum_{\mu  = 1}^\infty  {\frac{{H_{\mu ,i} }}{{\mu ^{3s - i} }}} }\\
&\qquad\qquad= \binom{2s-1}{s-1}\frac{{2\zeta (3s)}}{{2 + ( - 1)^s }}\\
&\qquad\qquad\quad + \frac{2}{{2( - 1)^s  + 1}}\sum_{i = 1}^{\left\lfloor {s/2} \right\rfloor } {\binom{2s-2i-1}{s-1}\zeta (2i)\zeta (3s - 2i)}\,,
\end{split} 
\]
from which we get
\begin{thm} For $s\in\mathbb{Z^+}$ holds
\[
\begin{split}
\sum_{i = 2}^s &{\binom{2s-i-1}{s-1}\sum_{\mu  = 1}^\infty  {\frac{{H_{\mu ,i} }}{{\mu ^{3s - i} }}} }\\
&\quad\qquad= \binom{2s-1}{s-1}\frac{{2\zeta (3s)}}{{2 + ( - 1)^s }}\\
&\qquad\qquad - \binom{2s-2}{s-1}\frac{{(3s + 1)\zeta (3s)}}{2}\\
&\quad\qquad\qquad + \frac{1}{2}\binom{2s-2}{s-1}\sum_{i = 1}^{3s - 3} {\zeta (3s - i - 1)\zeta (i + 1)}\\ 
&\qquad\qquad\qquad + \frac{2}{{2( - 1)^s  + 1}}\sum_{i = 1}^{\left\lfloor {s/2} \right\rfloor } {\binom{2s-2i-1}{s-1}\zeta (2i)\zeta (3s - 2i)}\,.
\end{split}
\]
\end{thm}

\bigskip

When $t>1$ in~\eqref{equ.ylv3nb4} and we use Theorem~\ref{th.nj1j516} to write $E(1,2s+t)$ we have
\begin{thm}\label{thm.bwgw7lo}
For $s,\,t-1\in\mathbb{Z^+}$ holds
\begin{equation}\label{equ.ge54qm8}
\begin{split}
&\sum_{i = 1}^{t - 1} {\left[ {\binom {s+i-1}s\frac{{2s + i - 1}}{{s + i - 1}}\sum_{\mu  = 1}^\infty  {\frac{{H_{\mu ,t - i + 1} }}{{\mu ^{2s + i} }}} } \right]}\\ 
&\quad= \sum_{i = 1}^{t - 1} { {\binom {s+i-1}s\frac{{2s + i - 1}}{{s + i - 1}}\zeta (t - i + 1)\zeta (2s + i)} }\\ 
&\qquad + ( - 1)^{s + 1} \sum_{i = 0}^{s - 1} {( - 1)^i \binom {t+i-1}{t-1}\frac{{t + 2i - 1}}{{t + i - 1}}\zeta (s - i + 1)\zeta (s + i + t)}\\
&\quad\qquad + \frac{1}{2}\frac{{(2s + t - 1)}}{{(s + t - 1)}}\binom {s+t-1}{t-1}\sum_{i = 1}^{2s + t - 2} {\zeta (i + 1)\zeta (2s + t - i)}\\ 
&\qquad\qquad - \frac{1}{2}\frac{{(2s + t - 1)(2s + t + 2)}}{{(s + t - 1)}}\binom {s+t-1}{t-1}\zeta (2s + t + 1)\,.
\end{split}
\end{equation}
or, equivalently,
\begin{equation}\label{equ.r81stqo}
\begin{split}
&\sum_{i = 2}^t {\left[ {\binom {s+t-i}s\frac{{2s + t - i}}{{s + t - i}}\sum_{\mu  = 1}^\infty  {\frac{{H_{\mu ,i} }}{{\mu ^{2s + t - i + 1} }}} } \right]}\\ 
&\quad= (-1)^{t-s}\sum_{i = t - s + 2}^{t + 1} { (-1)^i{\binom {2t-i}{t-1}\frac{{3t - 2i + 1}}{{2t - i}}\zeta (s - t + i)\zeta (s+2t-i+1)} }\\ 
&\qquad + \sum_{i = 2}^t {\binom {s+t-i}s\frac{{2s + t - i}}{{s + t - i}}\zeta (i)\zeta (2s + t -i + 1)}\\
&\quad\qquad + \frac{1}{2}\frac{{(2s + t - 1)}}{{(s + t - 1)}}\binom {s+t-1}{t-1}\sum_{i = 1}^{2s + t - 2} {\zeta (i + 1)\zeta (2s + t - i)}\\ 
&\qquad\qquad - \frac{1}{2}\frac{{(2s + t - 1)(2s + t + 2)}}{{(s + t - 1)}}\binom {s+t-1}{t-1}\zeta (2s + t + 1)\,.
\end{split}
\end{equation}

\end{thm}
If we set $t=3$ in identity~\eqref{equ.ge54qm8} we obtain
\begin{cor}
For $s\in\mathbb{Z^+}$ holds
\[
\begin{split}
2\sum_{\mu  = 1}^\infty  {\frac{{H_{\mu ,3} }}{{\mu ^{2s + 1} }}} \, +\, &(2s + 1)\sum_{\mu  = 1}^\infty  {\frac{{H_{\mu ,2} }}{{\mu ^{2s + 2} }}}\\
&\quad = 2\zeta (3)\zeta (2s + 1)\\
&\qquad + \frac{{( - 1)^{s + 1} }}{2}(s + 1)^2 \zeta (s + 2)^2\\
&\qquad\quad + ( - 1)^{s + 1} \sum_{i = 1}^{s-1} {( - 1)^i (s - i + 1)(s + i + 1)\zeta (s - i + 2)\zeta (s + i + 2)}\,. 
\end{split} 
\]

In particular (see also~\cite{flajolet}, page 23),
\[
2\sum_{\mu  = 1}^\infty  {\frac{{H_{\mu ,3} }}{{\mu ^5 }}}  + 5\sum_{\mu  = 1}^\infty  {\frac{{H_{\mu ,2} }}{{\mu ^6 }}}  =  - \frac{9}{2}\zeta (4)^2 + 10\zeta (3)\zeta (5)
\]
and
\begin{equation}\label{equ.bdgdtr}
2\sum_{\mu  = 1}^\infty  {\frac{{H_{\mu ,3} }}{{\mu ^7 }}}  + 7\sum_{\mu  = 1}^\infty  {\frac{{H_{\mu ,2} }}{{\mu ^8 }}}  =  - 15\zeta (4)\zeta(6) + 14\zeta (3)\zeta (7) + 8\zeta (5)^2\,. 
\end{equation}
\end{cor}
Combining~\eqref{equ.bdgdtr} and the identity of Corollary~\ref{thm.ab4k60c} we get
\[
\sum_{\mu  = 1}^\infty  {\frac{{H_{\mu ,3} }}{{\mu ^7 }}}  + \sum_{\mu  = 1}^\infty  {\frac{{H_{\mu ,4} }}{{\mu ^6 }}}  =  3\zeta (4)\zeta (6) - \zeta (5)^2  - \frac{{\zeta (10)}}{5}\,.
\]

Setting $t=s-1$ in~\eqref{equ.r81stqo} yields
\begin{cor}
For $s-2\in\mathbb{Z^+}$ holds
\[
\begin{split}
&\sum_{i = 2}^{s - 1} {\left\{ {\binom{2s-i-1}s\frac{{3s - i - 1}}{{2s - i - 1}}\sum_{\mu  = 1}^\infty  {\frac{{H_{\mu ,i} }}{{\mu ^{3s - i} }}} } \right\}}\\
&\quad = \sum_{i = 2}^{s + 1} {( - 1)^i \binom{2s-i-1}{s-2}\frac{{3s - 2i}}{{2s - i - 1}}\zeta (i)\zeta (3s - i)}\\
&\qquad + \sum_{i = 2}^{s - 1} {\binom{2s-i-1}s\frac{{3s - i - 1}}{{2s - i - 1}}\zeta (i)\zeta (3s - i)}\\
&\quad\qquad + \frac{1}{4}\frac{{(3s - 2)}}{{(s - 1)}}\binom{2s-2}{s-2}\sum_{i = 2}^{3s - 2} {\zeta (i)\zeta (3s - i)}\\
&\qquad\qquad - \frac{1}{4}\frac{{(3s - 2)(3s + 1)}}{{(s - 1)}}\binom{2s-2}{s-2}\zeta (3s)\,.
\end{split}
\]
\end{cor}
In particular,
\[
\begin{split}
9\sum_{\mu  = 1}^\infty  {\frac{{H_{\mu ,2} }}{{\mu ^{10} }}}  + 2\sum_{\mu  = 1}^\infty  {\frac{{H_{\mu ,3} }}{{\mu ^9 }}}&=50\zeta (2)\zeta (10) + 18\zeta (3)\zeta (9)\\
&\qquad\qquad + 29\zeta (4)\zeta (8) + 24\zeta (5)\zeta (7)\\
&\quad\qquad\qquad + \frac{{25}}{2}\zeta (6)^2  - \frac{{325}}{2}\zeta (12)\,.
\end{split}
\]
Setting $t=s$ in~\eqref{equ.r81stqo} gives
\begin{cor}
For $s-1\in\mathbb{Z^+}$ holds
\[
\begin{split}
&\sum_{i = 2}^{s} {\left\{ {\binom{2s-i}s\frac{{3s - i}}{{2s - i}}\sum_{\mu  = 1}^\infty  {\frac{{H_{\mu ,i} }}{{\mu ^{3s - i + 1} }}} } \right\}}\\
&\quad = \sum_{i = 2}^{s + 1} {( - 1)^i \binom{2s-i}{s-1}\frac{{3s - 2i + 1}}{{2s - i}}\zeta (i)\zeta (3s - i+1)}\\
&\qquad + \sum_{i = 2}^{s} {\binom{2s-i}s\frac{{3s - i}}{{2s - i}}\zeta (i)\zeta (3s - i+1)}\\
&\quad\qquad + \frac{1}{2}\frac{{(3s - 1)}}{{(2s - 1)}}\binom{2s-1}{s-1}\sum_{i = 2}^{3s - 1} {\zeta (i)\zeta (3s - i + 1)}\\
&\qquad\qquad - \frac{1}{2}\frac{{(3s - 1)(3s + 2)}}{{(2s - 1)}}\binom{2s-1}{s-1}\zeta (3s+1)\,.
\end{split}
\]
\end{cor}

Interchanging $\mu$ and $\nu$ (purely for notational consistency) in the identity of Lemma~\ref{thm.eclzyl6} and taking limit $\nu\to\infty$, we obtain
\[
\begin{split}
&\sum_{i = 1}^n {\binom{n+s-i-1}{s-1}\sum_{\mu  = 1}^\infty  {\frac{{H_{\mu ,i} }}{{\mu ^{n + s - i} }}} }\\
&\quad  + \sum_{i = 1}^s {\binom{n+s-i-1}{n-1}\sum_{\mu  = 1}^\infty  {\frac{{H_{\mu ,i} }}{{\mu ^{n + s - i} }}} }\\
&\qquad= \zeta (n)\zeta (s)+\binom{n+s}n\zeta(n+s)\,,
\end{split}
\]
from which we get

\begin{thm}
For $n-1,s-1\in\mathbb{Z^+}$ holds
\[
\begin{split}
&\sum_{i = 2}^{n-1} {\binom{n+s-i-1}{s-1}\sum_{\mu  = 1}^\infty  {\frac{{H_{\mu ,i} }}{{\mu ^{n + s - i} }}} }\\
&\quad  + \sum_{i = 2}^{s-1} {\binom{n+s-i-1}{n-1}\sum_{\mu  = 1}^\infty  {\frac{{H_{\mu ,i} }}{{\mu ^{n + s - i} }}} }\\
&\qquad=\zeta (n + s)\left[ \binom{n+s}n-(n+s+1)\binom{n+s-2}{s-1}-1\right]\\
&\qquad\qquad+\binom{n+s-2}{s-1}\sum_{i=2}^{n+s-2}{\zeta(i)\zeta(n+s-i)}\,.
\end{split}
\]
In particular, setting $n=s$ recovers the identity of Theorem~\ref{thm.xzc9oni}.
\end{thm}

\section{Summary}\label{sec.summary}
Salient aspects of relationships between generalized harmonic numbers, Euler sums and Tornheim series were discussed in this paper. 

\bigskip

We showed that every linear Euler sum can be decomposed into a linear combination of Tornheim double series of the same weight, that is,
\[
E(n,r)=\sum_{p=1}^n {T(r-1,n-p+1,p)}
\]
and
\[
E(n,r)=\sum_{p=1}^n{(-1)^{p-1}\binom npT(r-p,n,p)}
\,.
\]
Evaluation of certain sums in terms of generalized Harmonic numbers and the Riemann zeta function were obtained. Examples include
\[
\begin{split}
\sum_{\nu  = 1}^\infty  {\frac{1}{{\nu ^s (\nu  + \mu )^t }}}  &= \sum_{i = 0}^{s - 2} {\binom{t+i-1}i\frac{{( - 1)^i \zeta (s - i)}}{{\mu ^{t + i} }}}\\
&\quad+ ( - 1)^s\sum_{i = 0}^{t - 2} {\binom{s+i-1}i\frac{{ \left[ {\zeta (t - i) - H_{\mu ,t - i} } \right]}}{{\mu ^{s + i} }}}\\
&\qquad + ( - 1)^{s - 1} \binom{s+t-2}{s-1}\frac{{H_\mu  }}{{\mu ^{s + t - 1} }}\,,
\end{split}
\]

\[
\sum_{p = 1}^n {\left\{ {( - 1)^{p - 1} \sum_{\nu  = 1}^N {\frac{\mu }{{\nu ^{n - p + 1} (\mu  - \nu )^p }}} } \right\}}=H_{N,n}  - ( - 1)^n \left[ {H_{\mu  - 1,n}  - H_{\mu  - N - 1,n} } \right]\,,
\]

\[
\begin{split}
&( - 1)^{m - 1} \sum_{p = 0}^n {\left\{ {( - 1)^p\binom mp \sum_{\nu  = 1}^\infty {\frac{{\mu ^{m - p} }}{{\nu ^m (\nu  + \mu )^{n - p + 1} }}} } \right\}}\\
&\qquad= H_{\mu ,n + 1}  + ( - 1)^n \sum_{i = n+1}^{m - 1} {\left\{ {( - 1)^i\binom in \mu^{i-n}\zeta({i + 1}) } \right\}}\,,
\end{split}
\]

\[
\begin{split}
&( - 1)^m \sum_{p = 0}^n {\left\{ {\binom{m+p-1}p\sum_{\nu  = 1}^\infty  {\frac{{\mu ^{m + 1} }}{{\nu ^{m + p} (\nu  + \mu )^{n - p + 1} }}} } \right\}}\\
&\qquad =  - \mu H_{\mu ,n + 1}  + \sum_{i = 2}^m {( - 1)^i \binom{i+n-1}n\mu ^i \zeta (n + i)}
\end{split}
\]
and
\[
\begin{split}
&\sum_{p = 0}^n {\left\{ {\binom{m+p}{m}\sum_{\nu  = 1}^\infty {\frac{{\mu ^{m+1} }}{{\nu ^{n-p + 1} (\nu  + \mu )^{m + p +1} }}} } \right\}}\\
&\quad\qquad=\sum_{i = n}^{m+n} {\binom in\mu ^{i - n} H_{\mu ,i + 1} }  -\sum_{i = n+1}^{m+n} {\binom in\mu ^{i - n} \zeta(i + 1) }\,.
\end{split}
\]
Using the functional relations alone (see Section~\ref{sec.functional}), it was already possible to derive various evaluations of the Tornheim series in a much more effortless manner than is found in earlier works on the subject.

\bigskip

We extended previously known results concerning linear Euler sums by deriving new closed form evaluations in Riemann zeta values of various Euler sums. Specifically, we obtained, among other results,
\[
\begin{split}
2\sum_{\nu  = 1}^\infty  {\frac{{H_\nu  }}{{(\nu  + \mu )^s }}} &= 2H_{\mu  - 1} \left( {\zeta (s) - H_{\mu  - 1,s} } \right)\\
&\quad + s\left( {\zeta (s + 1) - H_{\mu  - 1,s + 1} } \right)\\
&\qquad- \sum_{i = 1}^{s - 2} {\left\{ {\left( {\zeta (i + 1) - H_{\mu  - 1,i + 1} } \right)\left( {\zeta (s - i) - H_{\mu  - 1,s - i} } \right)} \right\}}\,,
\end{split}
\]
\[
\begin{split}
&2\sum_{\nu  = 1}^\infty  {\frac{{H_{\nu ,n} }}{{(\nu  + \mu )^n }}}  = H_{\mu ,2n}  - \zeta (2n) - H_{\mu ,n}^2  + \zeta (n)^2\\ 
&\qquad\quad + ( - 1)^{n - 1} 4\sum_{j = 1}^{\left\lfloor {n/2} \right\rfloor } {\binom{2n-2j-1}{n-1}\zeta (2j)H_{\mu  - 1,2n - 2j} }\\
&\qquad - ( - 1)^{n - 1} 2\sum_{j = 1}^n {\binom{2n-j-1}{n-1}( - 1)^j H_{\mu  - 1,2n - j} H_{\mu ,j} }\\ 
&\quad\qquad - ( - 1)^{n - 1} 2\sum_{j = 1}^n {\left\{ {\binom{2n-j-1}{n-1}( - 1)^j \sum_{k = 1}^{2n - j} {\left[ {\binom{2n-k-1}{j-1}\sum_{i = 1}^\mu  {\frac{{H_{i - 1,k} }}{{i^{2n - k} }}} } \right]} } \right\}}\\ 
&\qquad\qquad - ( - 1)^{n - 1} 2\sum_{j = 1}^n {\left\{ {\binom{2n-j-1}{n-1}( - 1)^j \sum_{k = 1}^j {\left[ {\binom{2n-k-1}{2n-j-1}\sum_{i = 1}^\mu  {\frac{{H_{i - 1,k} }}{{i^{2n - k} }}} } \right]} } \right\}}\,, 
\end{split}
\]
\[
\begin{split}
\mu \sum_{\nu  = 1}^\infty  {\frac{{H_{\nu ,n} }}{{\nu (\nu  + \mu )}}} &= ( - 1)^{n - 1} \sum_{j = 1}^{n - 1} {( - 1)^j H_{\mu  - 1,n - j}\, \zeta (j + 1)}\\ 
&\quad + ( - 1)^{n - 1} \sum_{k = 1}^n \left\{( - 1)^k{\sum_{j = 1}^{n - k + 1} \left[{ { \binom{n-j}{k-1}\sum_{i = 1}^{\mu} {\frac{{H_{i - 1,j} }}{{i^{n - j + 1} }}} }} \right]}\right\}\\ 
&\qquad + ( - 1)^{n - 1} \sum_{k = 1}^n\left\{ ( - 1)^k{\sum_{j = 1}^k\left[ {{\binom{n-j}{n-k}\sum_{i = 1}^{\mu} {\frac{{H_{i - 1,j} }}{{i^{n - j + 1} }}} }} \right]}\right\}\\ 
&\quad\qquad + \left( {1 + ( - 1)^{n - 1} } \right)H_\mu  H_{\mu  - 1,n}\\
&\qquad\qquad - \sum_{i = 1}^{\mu  - 1} {\frac{H_i}{i^n} }  + \zeta (n + 1)
\end{split}
\]
and
\[
\begin{split}
&2\sum_{\nu  = 1}^\infty  {\frac{{H_\nu  }}{{\nu ^s (\mu  + \nu )^t }}}  = ( - 1)^{s - 1} \binom{s+t-2}{s-1}\frac{{H_{\mu  - 1}^2  + H_{\mu  - 1,2}  + 2\zeta (2)}}{{\mu ^{s + t - 1} }}\\
&\qquad\qquad + \sum_{i = 0}^{s - 2} {\binom{t+i-1}i\frac{{( - 1)^i (s - i + 2)\zeta (s - i + 1)}}{{\mu ^{t + i} }}}\\ 
&\qquad\qquad\quad - \sum_{i = 0}^{s - 2} {\left\{ {\binom{t+i-1}i\frac{{( - 1)^i }}{{\mu ^{t + i} }}\sum_{j = 1}^{s - i - 2} {\zeta (j + 1)\zeta (s - i - j)} } \right\}}\\ 
&\qquad\quad + 2( - 1)^s H_{\mu  - 1} \sum_{i = 0}^{t - 2} {\binom{s+i-1}i\frac{{\left( {\zeta (t - i) - H_{\mu  - 1,t - i} } \right)}}{{\mu ^{s + i} }}}\\
&\quad + ( - 1)^s \sum_{i = 0}^{t - 2} {\binom{s+i-1}i\frac{{(t - i)\left( {\zeta (t - i + 1) - H_{\mu  - 1,t - i + 1} } \right)}}{{\mu ^{s + i} }}}\\
&\qquad - ( - 1)^s \sum_{i = 0}^{t - 2} {\left\{ {\binom{s+i-1}i\frac{1}{{\mu ^{s + i} }}\sum_{j = 1}^{t - i - 2} {\left( {\zeta (j + 1) - H_{\mu  - 1,j + 1} } \right)\left( {\zeta (t - i - j) - H_{\mu  - 1,t - i - j} } \right)} } \right\}}\,. 
\end{split}
\]
Finally we derived certain combinations of linear Euler sums that evaluate to zeta values, for example,
\[
\begin{split}
&\sum_{i = 2}^t {\left[ {\binom {s+t-i}s\frac{{2s + t - i}}{{s + t - i}}\sum_{\mu  = 1}^\infty  {\frac{{H_{\mu ,i} }}{{\mu ^{2s + t - i + 1} }}} } \right]}\\ 
&\quad= (-1)^{t-s}\sum_{i = t - s + 2}^{t + 1} { (-1)^i{\binom {2t-i}{t-1}\frac{{3t - 2i + 1}}{{2t - i}}\zeta (s - t + i)\zeta (s+2t-i+1)} }\\ 
&\qquad + \sum_{i = 2}^t {\binom {s+t-i}s\frac{{2s + t - i}}{{s + t - i}}\zeta (i)\zeta (2s + t -i + 1)}\\
&\quad\qquad + \frac{1}{2}\frac{{(2s + t - 1)}}{{(s + t - 1)}}\binom {s+t-1}{t-1}\sum_{i = 1}^{2s + t - 2} {\zeta (i + 1)\zeta (2s + t - i)}\\ 
&\qquad\qquad - \frac{1}{2}\frac{{(2s + t - 1)(2s + t + 2)}}{{(s + t - 1)}}\binom {s+t-1}{t-1}\zeta (2s + t + 1)
\end{split}
\]
and
\[
\begin{split}
&\sum_{i = 2}^{n-1} {\binom{n+s-i-1}{s-1}\sum_{\mu  = 1}^\infty  {\frac{{H_{\mu ,i} }}{{\mu ^{n + s - i} }}} }\\
&\quad  + \sum_{i = 2}^{s-1} {\binom{n+s-i-1}{n-1}\sum_{\mu  = 1}^\infty  {\frac{{H_{\mu ,i} }}{{\mu ^{n + s - i} }}} }\\
&\qquad=\zeta (n + s)\left[ \binom{n+s}n-(n+s+1)\binom{n+s-2}{s-1}-1\right]\\
&\qquad\qquad+\binom{n+s-2}{s-1}\sum_{i=2}^{n+s-2}{\zeta(i)\zeta(n+s-i)}\,.
\end{split}
\]

\medskip

\end{document}